\documentclass[12pt,leqno,a4paper]{article}
\usepackage{amssymb}
\usepackage{amsmath}
\usepackage{amsthm}
\usepackage{color}
\usepackage[pdfborder={0 0 0}]{hyperref}
\usepackage[shortlabels]{enumitem}
\usepackage{mathrsfs}							
\usepackage{bbm}									
\usepackage{esint}								

\newtheorem{theorem}{Theorem}
\newtheorem{proposition}[theorem]{Proposition}
\newtheorem{lemma}[theorem]{Lemma}
\newtheorem{corollary}[theorem]{Corollary}

\theoremstyle{remark}

\theoremstyle{definition}
\newtheorem{remark}[theorem]{Remark}
\newtheorem{definition}[theorem]{Definition}

\numberwithin{equation}{section}
\numberwithin{theorem}{section}

\newcommand\set[1]{\left\{\,#1\,\right\}}
\newcommand\abs[1]{\left|#1\right|}

\newcommand\norm[1]{\left\Vert#1\right\Vert}

\newcommand{\R}{\mathbb{R}}

\newcommand{\T}{\mathbb{T}}
\renewcommand{\P}{\mathbb{P}}

\newcommand{\cC}{{\mathcal C}}

\newcommand{\cW}{{\mathcal W}}

\DeclareMathOperator{\id}{id}
\DeclareMathOperator{\supp}{supp}

\DeclareMathOperator{\tr}{tr}

\DeclareMathOperator{\divv}{div}

\DeclareMathOperator{\Lip}{Lip}

\DeclareMathOperator*{\essinf}{ess\,inf}
\DeclareMathOperator*{\esssup}{ess\,sup}

\begin{document}

\title{On the energy-constrained optimal mixing problem for one-dimensional initial configurations}
\author{Bj\"orn Gebhard}
\date{}
\maketitle

\begin{abstract}
We consider the problem of mixing a passive scalar in a periodic box by incompressible vector fields subject to a fixed energy constraint. In that setting a lower bound for the time in which perfect mixing can be achieved has been given by Lin, Thiffeault, Doering \cite{Lin_Thiffeault_Doering_2011}. While examples by Depauw \cite{Depauw} and Lunasin et al. \cite{Lunasin_etal_2012} show that perfect mixing in finite time is indeed possible, the question regarding the sharpness of the lower bound from \cite{Lin_Thiffeault_Doering_2011} remained open. In the present article we give a negative answer for the special class of initial configurations depending only on one spatial coordinate. 
The new lower bound holds true for distributional solutions satisfying only the uniform energy constraint for the velocity field and a weak compatibility condition for the passive scalar coming from the transport equation. In that weak setting we also provide an example for which the new bound is sharp. 
As a new ingredient in the investigation of optimal mixing we utilize the convex hull inequalities of the transport equation with constraints when seen as a differential inclusion.
\end{abstract}

\section{Introduction}

Let $\T^d$, $d\geq 2$ be the $d$-dimensional flat torus with side length $L>0$. On $\T^d$ we consider the transport equation
\begin{align}\label{eq:transport}
\partial_t\rho +v\cdot\nabla \rho=0
\end{align}
with respect to a fixed non-constant essentially bounded initial distribution $\rho_0$ and divergence-free velocity fields $v$. Passing from $\rho_0\in L^\infty(\T^d)$ non-constant to $\tilde{\rho}_0$ defined by
\begin{align*}
\tilde\rho_0=\frac{\rho_0-\fint\rho_0}{\norm{\rho_0-\fint\rho_0}_{L^\infty(\T^d)}},
\end{align*} 
we will without loss of generality assume that the initial datum satisfies
\begin{align}\label{eq:rho_0_conditions}
\fint\rho_0=0,\quad \norm{\rho_0}_{L^\infty(\T^d)}= 1.
\end{align}
Here $\fint\rho_0$ denotes the average of $\rho_0$, i.e.
\[
\fint \rho_0:=L^{-d}\int_{\T^d}\rho_0(x)\:dx.
\]

We are interested in incompressible velocity fields $v$ that are smooth on some (possibly unbounded) time interval $[0,T)\subset [0,\infty)$, satisfy the energy constraint 
\begin{align}\label{eq:energy_constraint}
\int_{\T^d}\abs{v(t,x)}^2\:dx\leq E,\quad t\in[0,T)
\end{align} 
for some fixed $E>0$, and that perfectly mix the initial distribution $\rho_0$ as $t\rightarrow T$. The latter means that the (unique) solution $\rho(t,x)$ of \eqref{eq:transport} with initial condition $\rho_0$ satisfies
\begin{align}\label{eq:mixing_at_time_T}
\lim_{t\rightarrow T}\norm{\rho(t,\cdot)}_{H^{-1}(\T^d)}=0.
\end{align}
The $H^{-1}$-norm is a frequently used measure for mixing called the functional mixing scale. It has been shown that decay of any negative Sobolev norm $\norm{\cdot}_{H^{-s}}$ to $0$ is equivalent to mixing in the sense of ergodic theory, see \cite{mathew_mezic_petzold} for $s=1/2$ and \cite{Lin_Thiffeault_Doering_2011} for general $s>0$. In view of that we equivalently could have replaced \eqref{eq:mixing_at_time_T} just by weak $L^2$-convergence $\rho(t,\cdot)\rightharpoonup 0$ as $t\rightarrow T$. However, we directly stated the condition of perfect mixing in terms of $\norm{\cdot}_{H^{-1}}$, because we will also use it to evaluate the progress of mixing before $T$. For other mixing scales and their comparison to the functional scale we refer to \cite{Bressan,Lunasin_etal_2012,Thiffeault_survey_mixing_scales,Zillinger_comparison_mixing_scales}.

To be precise we collect the required conditions in the following definition.
\begin{definition}\label{def:perfectly_mixing_velocity_fields}
Let $\rho_0\in L^\infty(\T^d)$ satisfy \eqref{eq:rho_0_conditions} and $T\in(0,\infty]$. We say that a velocity field $v$ is perfectly mixing $\rho_0$ at time $T$ if and only if $v\in \cC^\infty([0,T)\times\T^d;\R^d)$, $v$ satisfies the energy constraint \eqref{eq:energy_constraint}, $\divv v(t,\cdot)=0$ for all $t\in[0,T)$, and the function $\rho\in L^\infty((0,T)\times\T^d)$ defined through the requirement
\begin{align}\label{eq:distr_solution}
\int_0^T\int_{\T^d}\rho\partial_t\psi+\rho v\cdot\nabla \psi\:dx\:dt+\int_{\T^d}\rho_0\psi(0,\cdot)\:dx=0
\end{align}
for all $\psi\in\cC^\infty_c([0,T)\times\T^d)$ satisfies \eqref{eq:mixing_at_time_T}.
\end{definition}
We remark that \eqref{eq:distr_solution} states that $\rho$ is a weak solution of \eqref{eq:transport} considered with $v$. In fact, by the assumed regularity for $v$ there exists only one such solution, which is given in terms of transport through the Lagrangian flow induced by $v$. Due to the regularity of $v$ the time $T$ is the first time for which \eqref{eq:mixing_at_time_T} holds true. However, our analysis will not rely on the requested smoothness of $v$. It also holds true in a weak setting where uniqueness of solutions is lost, see Section \ref{sec:weak_setting}.

Note also that the discussion in Appendix \ref{sec:appendix_lipschitz} implies that $\norm{\rho(t,\cdot)}_{H^{-1}(\T^d)}$ is well-defined for any $t\in[0,T)$ and continuous with respect to time.

The energy constraint optimal mixing problem consists of finding velocity fields $v$ that perfectly mix $\rho_0$ in the least possible time. Next we abstractly define the optimal mixing time, which we seek to characterize in a more specific way.
\begin{definition}\label{def:optimal_mixing_time}
Given $\rho_0\in L^\infty(\T^d)$ satisfying \eqref{eq:rho_0_conditions}, the optimal mixing time $T_{mix}(\rho_0)$ is defined as the infimum of all $T>0$ with the property that there exists a velocity field $v$ that perfectly mixes $\rho_0$ at time $T$.
\end{definition}

\subsection{Lower bounds}

Apriori $T_{mix}(\rho_0)\in[0,+\infty]$. However, examples for perfect mixing in finite time have been given by Depauw \cite{Depauw} and Lunasin et al. \cite{Lunasin_etal_2012}.

On the other hand a lower bound for the optimal mixing time has been derived by Lin, Thiffeault and Doering \cite{Lin_Thiffeault_Doering_2011}, and also in \cite{Lunasin_etal_2012}. It states that
\begin{align}\label{eq:first_lower_bound}
T_{mix}(\rho_0)\geq E^{-\frac{1}{2}}\norm{\rho_0}_{H^{-1}(\T^d)}.
\end{align}
Indeed, if we suppose that $\rho$ is the solution induced by a velocity field $v$ that is perfectly mixing $\rho_0$ at time $T$ and $\varphi$ is defined via $\Delta\varphi =\rho$, then we (for now formally) compute
\begin{align}\begin{split}\label{eq:computation_first_lower_bound}
    \frac{d}{dt}\norm{\rho(t,\cdot)}_{H^{-1}(\T^d)}^2&=\frac{d}{dt}\norm{\nabla\varphi(t,\cdot)}_{L^2(\T^d)}^2=-2\int_{\T^d}\rho(t,x)v(t,x)\cdot\nabla\varphi(t,x)\:dx\\
    &\geq -2\sqrt{E}\norm{\rho(t,\cdot)}_{H^{-1}(\T^d)},
    \end{split}
\end{align}
and thus
\begin{align*}
0=\lim_{t\rightarrow T}\norm{\rho(t,\cdot)}_{H^{-1}(\T^d)}\geq \norm{\rho_0}_{H^{-1}(\T^d)}-\sqrt{E}T,
\end{align*}
implies \eqref{eq:first_lower_bound}. We remark that this computation can be made rigorous, for instance via Lemma \ref{lem:mixing_norm_Lipschitz} in Appendix \ref{sec:appendix_lipschitz}, i.e. the lower bound \eqref{eq:first_lower_bound} holds for any $\rho_0$ satisfying \eqref{eq:rho_0_conditions}.

As mentioned above this estimate can be found in the paper of Lin, Thiffeault, Doering \cite{Lin_Thiffeault_Doering_2011} and also in the aforementioned article of Lunasin et al. \cite{Lunasin_etal_2012}. Regarding the sharpness of the bound Lin, Thiffeault and Doering remark the following, \cite[p. 469]{Lin_Thiffeault_Doering_2011}:
\emph{``Whether or not this limiting mixing rate can actually be achieved, or even approached, by any suitably constrained stirring flow remains to be seen.''}

The motivation of the present paper is to shed some light onto that. For the special class of one-dimensional initial data it turns out that \eqref{eq:first_lower_bound} is in fact not sharp as the following statement shows.
\begin{theorem}\label{thm:lower_bound_1D}
Let $\rho_0\in L^\infty(\T^d)$ satisfy \eqref{eq:rho_0_conditions} and assume that $\rho_0$ depends only on one space-variable. Then 
$\norm{\rho_0}_{H^{-1}(\T^d)}\in [0, H_{max}]$ with $H_{max}:=L^{\frac{d+2}{2}}/\sqrt{48}$ and there holds
\begin{align}\label{eq:new_lower_bound_1D}
T_{mix}(\rho_0)\geq E^{-\frac{1}{2}}H_{max}\sqrt{2}S(\alpha(r_0)),
\end{align}
where $r_0:=\norm{\rho_0}_{H^{-1}(\T^d)}/H_{max}\in[0,1]$ and 
\begin{gather*}
S(\alpha):=\frac{1}{2}\left(\arcsin(\alpha)+\alpha\sqrt{1-\alpha^2}\right),\quad \alpha(r):=\left(1-\sqrt{1-r^2}\right)^\frac{1}{2}.
\end{gather*}
\end{theorem}
\begin{remark}\label{rem:lower_bound_is_improvement} This is indeed an improvement compared to the previous lower bound \eqref{eq:first_lower_bound}: Rewriting \eqref{eq:first_lower_bound} in terms of the ratio $r_0=\norm{\rho_0}_{H^{-1}(\T^d)}/H_{max}$ the gap between the right-hand sides of \eqref{eq:new_lower_bound_1D} and \eqref{eq:first_lower_bound} is given by 
\begin{align}\label{eq:difference_of_bounds}
E^{-\frac{1}{2}}H_{max}\left(\sqrt{2}S(\alpha(r_0))-r_0\right)=:E^{-\frac{1}{2}}H_{max}\delta(r_0).
\end{align}
Clearly $\delta(0)=0$. We will show that $\delta'(r_0)>0$ for $r_0\in(0,1)$. For that it is more convenient to invert $\alpha(r)$, which is increasing on $(0,1)$, and to consider $r_0$ as a function of $\alpha_0:=\alpha(r_0)$, i.e. $r_0=\alpha_0\sqrt{2-\alpha_0^2}$. Then
\begin{align*}
\frac{d}{d\alpha_0}\left(\sqrt{2}S(\alpha_0)-\alpha_0\sqrt{2-\alpha_0^2}\right)=\sqrt{2-2\alpha_0^2}-\sqrt{2-\alpha_0^2}+\frac{\alpha_0^2}{\sqrt{2-\alpha_0^2}}>0
\end{align*}
for $\alpha_0\in(0,1)$.
\end{remark}

\subsection{Weak setting}\label{sec:weak_setting}
We emphasize that the proof of Theorem \ref{thm:lower_bound_1D} is not relying on the smoothness of the velocity $v$. In fact, as Theorem \ref{thm:weak_setting} below shows, the lower bound holds for any energy admissible distributional solution as long as the density $\rho$ satisfies condition \eqref{eq:condition_rho_weak_setting} which, recall that $\norm{\rho_0}_{L^\infty(\T^d)}=1$, is a weak compatibility condition with the transport equation.
\begin{theorem}\label{thm:weak_setting}
Let $\rho_0$ be as in Theorem \ref{thm:lower_bound_1D} and $T>0$. Suppose that $\rho\in L^\infty((0,T)\times\T^d)$, $v\in L^\infty(0,T;L^2(\T^d;\R^d))$ with 
\begin{gather}\label{eq:condition_energy_weak_setting}
\norm{v}_{L^\infty(0,T;L^2(\T^d;\R^d))}\leq \sqrt{E},\\
\label{eq:condition_rho_weak_setting}
\norm{\rho}_{L^\infty((0,T)\times\T^d)}\leq 1
\end{gather}
solve the system 
\begin{align}\label{eq:equations_distributional}
\begin{cases}\partial_t\rho+\divv(\rho v)=0,&\text{on }(0,T)\times\T^d,\\
\divv v=0,&\text{on }(0,T)\times\T^d,\\
\rho(0,\cdot)=\rho_0,&\text{at }t=0
\end{cases}
\end{align}
in the sense of distributions. Suppose further that $\norm{\rho(t,\cdot)}_{H^{-1}(\T^d)}\rightarrow 0$ as $t\rightarrow T$. Then $T\geq E^{-\frac{1}{2}}H_{max}\sqrt{2}S(\alpha(r_0))$ with $H_{max},S,\alpha,r_0$ as in Theorem \ref{thm:lower_bound_1D}.
\end{theorem}

\begin{remark}
a) By distributional solution we mean that in addition to \eqref{eq:distr_solution} there holds 
\[
\int_0^T\int_{\T^d}v\cdot\nabla\psi\:dx\:dt=0
\] for all $\psi\in\cC^{\infty}([0,T]\times\T^d)$.

b) We also here point out that $t\mapsto \norm{\rho(t,\cdot)}_{H^{-1}(\T^d)}$ can assumed to be continuous in view of Appendix \ref{sec:appendix_lipschitz}.

c) Except for the statement that $\norm{\rho_0}_{H^{-1}(\T^d)}\leq H_{max}$, Theorem \ref{thm:lower_bound_1D} is a consequence of Theorem \ref{thm:weak_setting}.

d) We remark that Theorem \ref{thm:weak_setting} is false if condition \eqref{eq:condition_rho_weak_setting} is not imposed. This follows from the non-uniqueness results for the transport equation of Modena, Sz\'ekelyhidi \cite{Modena_Sz_AnnPDE_2018,Modena_Sz_CalcVar_2019} and Modena, Sattig \cite{Modena_Sattig_2020}. Indeed, applying for example \cite[Theorem 1.2]{Modena_Sattig_2020} to $p=2$, $\bar{u}\equiv 0$ and $\bar{\rho}(t,x)=\chi(t)\rho_0(x_1)$ with a smooth $\rho_0$ satisfying \eqref{eq:rho_0_conditions} and $\chi(t)$ smoothly transitioning from $1$ near $t=0$ to $0$ near $t=\delta$, $\delta>0$ arbitrarily small, one obtains solutions $\rho\in\cC^0([0,\delta];L^2(\T^d))$, $v\in \cC^0([0,\delta];L^2(\T^d))$ satisfying \eqref{eq:condition_energy_weak_setting} that are perfectly mixed at time $\delta$. In fact the velocity fields can have arbitrarily small energy $E$ and in addition enjoy certain Sobolev regularity.
\end{remark}

\subsection{Sharpness}
In the last part of our analysis we will look at a specific example. The computation in Remark \ref{rem:lower_bound_is_improvement} shows that the gap between the lower bound \eqref{eq:first_lower_bound} and the lower bound \eqref{eq:new_lower_bound_1D} is maximal for $r_0=\alpha_0=1$, i.e. for a ``maximally unmixed'' $\rho_0$. Lemma \ref{lem:estimate_L2_norm} states that this state is up to translation and permutation of coordinates uniquely given by
\begin{align}\label{eq:definition_hat_rho}
\hat{\rho}_0(x):=\begin{cases}
1,&x_1\in(-L/2,0),\\
-1,&x_1\in (0,L/2).
\end{cases}
\end{align}
Theorem \ref{thm:lower_bound_1D} therefore yields the estimate
\begin{equation}\label{eq:lower_bound_smooth_for_rho0_hat}
T_{mix}(\hat{\rho}_0)\geq E^{-\frac{1}{2}}H_{max}\frac{\sqrt{2}\pi}{4}=:\hat{T}_0.
\end{equation}
\begin{theorem}\label{thm:sharpness_weak_setting}
For any $\varepsilon>0$ there exist infinitely many distributional solutions $(\rho_\varepsilon,v_\varepsilon)$ of \eqref{eq:equations_distributional} on $(0,\hat{T}_0+\varepsilon)\times\T^d$ with initial data $\rho_0=\hat{\rho}_0$. The solutions satisfy \eqref{eq:condition_energy_weak_setting} with $T=\hat{T}_0+\varepsilon$ and $\abs{\rho}=1$ almost everywhere on $(0,\hat{T}_0+\varepsilon)\times\T^d$. Furthermore, $\norm{\rho_\varepsilon(t,\cdot)}_{H^{-1}(\T^d)}\rightarrow 0$ as $t\rightarrow \hat{T}_0+\varepsilon$.
\end{theorem}

Theorem \ref{thm:sharpness_weak_setting} shows that Theorem \ref{thm:weak_setting} is sharp for $\hat{\rho}_0$, meaning that the bound on the optimal mixing time for $\hat{\rho}_0$ is sharp if one only asks for the natural apriori bounds \eqref{eq:condition_energy_weak_setting}, \eqref{eq:condition_rho_weak_setting} and nothing else. We emphasize that this of course does not answer the question if there holds equality in \eqref{eq:lower_bound_smooth_for_rho0_hat} due to the lack of smoothness of $v$. On a related note we also like to mention that there is no Lagrangian picture for the mixing solutions of Theorem \ref{thm:sharpness_weak_setting}, meaning that the mixing does not occur through classical transport under a volume preserving flow, simply because the velocity fields do not induce well-defined (regular) Lagrangian flows. 

However, with respect to extending Theorem \ref{thm:sharpness_weak_setting} to a smooth, or at least somewhat smoother, setting we point out that our analysis in Section \ref{sec:sharpness_on_averaged_level} gives a precise description of the vertical averages of any solution, smooth or non-smooth, which would realize the mixing time $\hat{T}_0$ for $\hat{\rho}_0$, see Section \ref{sec:sharpness} for further discussion.

\subsection{Outline of the strategy}
One of the key aspects in our analysis is to shift the focus from the actual solutions $(\rho,v)$ to averaged solutions $(\bar{\rho},\bar{v})$. In the context of fluid dynamics these are typically called subsolutions and they are utilized to construct turbulently mixing solutions emanating from classical hydrodynamic instabilities like Kelvin-Helmholtz \cite{GK_EE,Mengual_Sz_sheets,Sz_KH}, Rayleigh-Taylor \cite{GHK_LAP,GK_Boussinesq,GKSz_RT}, Saffman-Taylor \cite{Arnaiz_Castro_Faraco,Castro_Cordoba_Faraco_invent,Castro_Faraco_G,Castro_Faraco_Mengual_turned,Foerster_Sz,Mengual_different_mobilities,Noisette_Sz,Sz_ipm} and recently also for circular vortex filaments \cite{Gancedo_HT_Mengual_circular_filament}.

However, in our case we use vertical averaging (or the corresponding direction dictated by $\rho_0$) leaving the initial datum invariant and creating an artificial mixing at positive times. This way any lower bound on the $H^{-1}$-norm of $\bar{\rho}(t,\cdot)$ serves as a lower bound for $\norm{\rho(t,\cdot)}_{H^{-1}}$. This step and further properties of the averaged solutions are presented in Sections \ref{sec:vertical_averages}, \ref{sec:general_1D_subsols}.

Thereafter we investigate the decay rate of all possible vertically averaged solutions. This is done in two steps. First, in Section \ref{sec:stepest_descent} we translate computation \eqref{eq:computation_first_lower_bound}, which lead to the lower bound of \cite{Lin_Thiffeault_Doering_2011,Lunasin_etal_2012}, to the level of averaged solutions. This gives us an inequality of the form
\begin{equation}\label{eq:formal_decay_rate_inequality_1}
\frac{d}{dt}\norm{\bar{\rho}(t,\cdot)}^2_{H^{-1}}\geq -C F[\bar{\rho},\bar{\varphi}],
\end{equation}
where $C>0$ is an (explicit) constant depending on $L,E,d$ and $F[\bar{\rho},\bar{\varphi}]$ stands symbolically for an integral involving $\bar{\rho}$ and the potential $\bar{\varphi}$ defined through $\Delta\bar{\varphi}=\bar{\rho}$.

In order to get a closed inequality only involving $\norm{\bar{\rho}(t,\cdot)}_{H^{-1}}$ we investigate the variational problem to maximize $F[\bar{\rho},\bar{\varphi}]$ under a constraint prescribing the $H^{-1}$-norm of $\bar{\rho}$. The problem can be solved explicitly relying on two symmetrizations and convexity, see Section \ref{sec:variational_problem}. 

This way we have turned \eqref{eq:formal_decay_rate_inequality_1} into an ordinary differential inequality of the type
\[
\frac{d}{dt}\norm{\bar{\rho}(t,\cdot)}^2_{H^{-1}}\geq -Cf\big(\norm{\bar{\rho}(t,\cdot)}_{H^{-1}}\big)
\]
which is finally leading to \eqref{eq:new_lower_bound_1D}, see Section \ref{sec:analysis_of_ode}.

Section \ref{sec:sharpness} addresses the sharpness of the derived lower bound for the specific initial data \eqref{eq:definition_hat_rho}. The first part concerns sharpness on the level of averaged solutions. Remarkably it turns out that there is an explicit $\bar{\rho}(t,x)$ that satisfies \eqref{eq:formal_decay_rate_inequality_1} with equality for all $t$, while at each time it is also a maximizer for the corresponding constrained variational problem studied in Section \ref{sec:variational_problem}. Thus $\bar{\rho}$ realizes $T_{mix}(\hat{\rho}_0)$ on averaged level. This is done in Section \ref{sec:sharpness_on_averaged_level}. 

Finally, in order to get actual solutions at least in the weak setting having the same mixing time as $\bar{\rho}$ up to an arbitrary small error, i.e. in order to prove Theorem \ref{thm:sharpness_weak_setting}, we utilize convex integration in a typical Baire-category manner, see Section \ref{sec:convex_integration}.

\subsection{Background and related results}
The question of mixing (or lack thereof) in fluids arises in many fields of science and day-to-day life, we refer to the examples given in \cite{Alberti_Crippa_Mazzucato,CZ_Crippa_Iyer_Mazzucato_survey,Thiffeault_survey_mixing_scales}.

The here considered problem is part of a whole family of optimal mixing problems which ask to design incompressible velocity fields $v(t,x)$ subject to certain constraints such that a given initial configuration $\rho_0(x)$ gets mixed via \eqref{eq:transport} in the best possible way. Frequent constraints are a uniform in time bound on a (homogeneous) $W^{s,p}$-norm, where the energy-constrained ($s=0,p=2$), enstrophy-constrained ($s=1,p=2$) and palenstrophy-constrained ($s=2,p=2$) optimal mixing problems form the physically most relevant cases.

In the case $s<1$, $p\in[1,\infty]$, which includes the energy-constrained setting, uniqueness for solutions of $\eqref{eq:transport}$ is lost and thus perfect mixing in finite time is possible. As explained in the previous sections a first lower bound on the mixing time has been given in \cite{Lin_Thiffeault_Doering_2011} and examples with a finite mixing time can be found in \cite{Depauw,Lunasin_etal_2012}.

In contrast for $s=1$, $p\in[1,\infty]$ perfect mixing in finite time is excluded due to the DiPerna-Lions uniqueness result \cite{DiPernaLions}. In the case $p>1$ lower exponential bounds for different mixing scales have first been given in \cite{Crippa_DeLellis_lower_bound,Iyer_Kiselev_Xu_lower_bound,Seis_lower_bound} and then also in \cite{Brue_Nguyen_lower_bound,HSSS_lower_bound,Leger_lower_bound,Meyer_Seis_lower_bound}. First examples that realize an exponential decay of the mixing norm have been given in \cite{Alberti_Crippa_Mazzucato} for $p\leq\infty$ and \cite{Yao_Zlatos} for exponents $p$ below a certain value (the enstrophy constrained case $p=2$ is covered) but for any initial configuration. Further constructions emphasizing for instance universality (or lack thereof) of the rate with respect to the initial data or time periodicity of the used flow can be found in \cite{Bedrossian_Blumenthal_PS,Blumenthal_CotiZelati_Gvalani,Cooperman,Elgindi_Liss_Mattingly,Elgindi_Zlatos,MyersHill_Sturman_Wilson}. An exponential bound from below for the case $s=1$ and $p=1$, which is related to Bressan's conjecture \cite{Bressan}, is open. 
The same is true for divergence free vectorfields uniformly bounded in $BV$ where finite time mixing is exclude as well \cite{Ambrosio}. However, in \cite{Bianchini_Zizza} it has been shown that exponentially fast mixing $L^\infty_tBV_x$ vector fields are in fact dense with respect to $\norm{\cdot}_{L^1_tL^1_x}$. 

The case $s>1$ inherits the exponential lower bound from $s=1$. For the special class of cellular vector fields the bound could be improved to a polynomial rate \cite{Crippa_Schulze_2017,Crippa_Schulze_2023}. Exponentially mixing (and thus not cellular) examples for a certain range of $(s,p)$ have been constructed in \cite{Elgindi_Zlatos}. For general exponents $s>1$, $p\in[1,\infty]$ cellular examples with polynomial decay have been given in \cite{Alberti_Crippa_Mazzucato}.      

In other types of optimal mixing constraints the velocity field is required to satisfy an additional fluid pde subject to control terms through which the fluid can be manipulated, see for instance \cite{Mathew_etal_Stokes} for Stokes flows controlled by finite-dimensional forces and \cite{Hu_Wu_Navier_Stokes} for Navier-Stokes under the influence of a boundary control.

For further background and results about mixing and also its link to related topics like enhanced or anomalous dissipation we refer to the recent survey \cite{CZ_Crippa_Iyer_Mazzucato_survey} as a great starting point for reading, as well as the aforementioned articles \cite{Alberti_Crippa_Mazzucato,Thiffeault_survey_mixing_scales}.

\subsection{Further questions}
As mentioned earlier the question if the optimal mixing time for $\hat{\rho}_0$ can also be approached under the influence of smooth velocity fields, i.e. if there holds equality in \eqref{eq:lower_bound_smooth_for_rho0_hat}, remains open. It would also be interesting to see if the here followed approach through the analysis of an averaged picture can be extended to non-flat initial data and to other disciplines of optimal mixing described in the previous subsection. For instance to prescribe $\norm{v}_{L^\infty_tW^{s,p}_x}$ instead of $\norm{v}_{L^\infty_tL^2_x}$ or to ask $v$ to satisfy in addition a fluid equation with an external force through which the possibilities of stirring are restricted.

\section{Improved lower bound for one-dimensional initial data}\label{sec:improved_bound_1D_data}

In this section we will derive \eqref{eq:new_lower_bound_1D} for one-dimensional initial data relying on an investigation of the associated averaged evolution.
\begin{definition}\label{def:one_dimensional_initial_data}
We say that $\rho_0$ as in \eqref{eq:rho_0_conditions} is one-dimensional provided $\rho_0$ depends only on $x_1$, i.e. $\rho_0(x)=\rho_0(x_1)$.
\end{definition}

\subsection{Vertical averages}\label{sec:vertical_averages}

Let $\rho_0$ be one-dimensional and $\rho \in L^\infty((0,T)\times\T^d)$, $v\in L^{\infty}(0,T;L^2(\T^d;\R^d))$ be a distributional solution of \eqref{eq:equations_distributional} satisfying \eqref{eq:condition_energy_weak_setting}, \eqref{eq:condition_rho_weak_setting}. Of course the influence of $v$ will break the initially present symmetry of the density. We will evade this loss of symmetry by passing to corresponding space averages. 

For that we recast the transport equation together with the imposed constraints as a differential inclusion. We begin by introducing the following additional quantities
\begin{align*}
m:=\rho v\in L^\infty(0,T;L^2(\T^d;\R^d)),\quad e:=\abs{v}^2\in L^\infty(0,T;L^1(\T^d)).
\end{align*}
Then the tuple $z:=(\rho,v,m,e)\in L_{t,x}^\infty\times L^\infty_tL^2_x\times L^\infty_tL^2_x\times L^\infty_tL^1_x$ satisfies the linear equations
\begin{align}\label{eq:linear_equation}
\partial_t\rho+\divv m=0,\quad \divv v=0
\end{align}
with initial condition $\rho_0$ (the notion of solution is understood in analogy to \eqref{eq:distr_solution}), as well as the pointwise constraint
\begin{equation}\label{eq:nonlinear_constraint}
z(t,x)\in K\quad \text{for almost every }(t,x)\in(0,T)\times\T^d,
\end{equation}
where the set $K\subset\R\times\R^d\times\R^d\times\R$ is given by
\begin{equation}\label{eq:definition_of_K}
K:=\set{(\rho,v,m,e):\abs{\rho}\leq 1,~m=\rho v,~\abs{v}^2=e}.
\end{equation}
Here we have used \eqref{eq:condition_rho_weak_setting}. Furthermore, the energy constraint \eqref{eq:energy_constraint}, or \eqref{eq:condition_energy_weak_setting}, now reads 
\begin{equation}\label{eq:integral_e}
\int_{\T^d}e(t,x)\:dx\leq E
\end{equation}
for almost every $t\in(0,T)$.

For $x\in\T^d$ we write $(x_1,x')\in\T\times\T^{d-1}$ and define the vertically averaged tuple $\bar{z}=(\bar{\rho},\bar{v},\bar{m},\bar{e})$ by
\begin{equation}\label{eq:definition_average_of_z}
\bar{z}(t,x_1):=L^{-(d-1)}\int_{\T^{d-1}}z(t,x_1,x')\:dx'.
\end{equation}
Moreover, $K^{co}$ denotes the closed convex hull of $K$.
\begin{lemma}\label{lem:properties_of_averaged_solutions}
For any one-dimensional $\rho_0$ and any perfectly mixing $v$ the tuple $\bar z=(\bar{\rho},\bar v,\bar m,\bar e)\in L_{t,x}^\infty\times L^\infty_tL^2_x\times L^\infty_tL^2_x\times L^\infty_tL^1_x$ has the following properties:
\begin{enumerate}[(i)]
\item\label{itm:averaged_i} $(\bar{\rho},\bar{m}_1)$ satisfies $\partial_t\bar{\rho}+\partial_{x_1}\bar{m}_1=0$ on $(0,T)\times\T$, $\bar{\rho}(0,\cdot)=\rho_0$  in the sense of distributions,
\item\label{itm:averaged_ii} $\partial_{x_1}\bar{v}_1=0$ on $(0,T)\times\T^d$ in the sense of distributions,
\item\label{itm:averaged_iii} $L^{d-1}\int_{\T}\bar{e}(t,x_1)\:dx_1\leq E$ for almost every $t\in(0,T)$,
\item\label{itm:averaged_iv} for almost every $(t,x_1)\in(0,T)\times \T$ there holds $\bar{z}(t,x_1)\in K^{co}$,
\item\label{itm:averaged_v} there holds $\norm{\rho(t,\cdot)}_{H^{-1}(\T^d)}\geq \norm{\bar{\rho}(t,\cdot)}_{H^{-1}(\T^d)}=L^{\frac{d-1}{2}}\norm{\bar{\rho}(t,\cdot)}_{H^{-1}(\T)}$ for any $t\in[0,T)$.
\end{enumerate}
\end{lemma}
\begin{proof} The first two properties are direct consequences of the linearity of the equations \eqref{eq:linear_equation} and the symmetry of the initial data $\rho_0$. Property \ref{itm:averaged_iii} is immediate. Property \ref{itm:averaged_iv} is a consequence of \eqref{eq:nonlinear_constraint} in combination with any sort of averaging. Regarding \ref{itm:averaged_v} one first of all sees that if $\Delta \varphi(t,\cdot)=\bar{\rho}(t,\cdot)$ on $\T^d$, then $\varphi(t,x)=\varphi(t,x_1)$. Thus $\norm{\bar\rho(t,\cdot)}_{H^{-1}(\T^d)}=L^{\frac{d-1}{2}}\norm{\bar{\rho}(t,\cdot)}_{H^{-1}(\T)}$. Moreover, 
\begin{align*}
&\norm{\bar{\rho}(t,\cdot)}_{H^{-1}(\T^d)}^2=\int_{\T^d}(\partial_{x_1}\varphi(t,x_1))^2\:dx=-L^{d-1}\int_{\T}\bar{\rho}(t,x_1)\varphi(t,x_1)\:dx_1\\
&\hspace{22pt}=-\int_{\T^d}\rho(t,x)\varphi(t,x_1)\:dx\leq \sup_{\norm{\nabla \phi}_{L^2(\T^d)}=1}\int_{\T^d}\rho(t,x)\phi(x)\:dx \norm{\partial_{x_1}\varphi(t,\cdot)}_{L^2(\T^d)}\\
&\hspace{22pt}=\norm{\rho(t,\cdot)}_{H^{-1}(\T^d)}\norm{\bar{\rho}(t,\cdot)}_{H^{-1}(\T^d)},
\end{align*}
which finishes the proof of the lemma.
\end{proof}
Clearly \ref{itm:averaged_v} is important for us, as it allows to translate any lower bound on $\norm{\bar{\rho}(t,\cdot)}_{H^{-1}(\T)}$ to a lower bound on the mixing norm $\norm{\rho(t,\cdot)}_{H^{-1}(\T^d)}$. Note also that in \ref{itm:averaged_iv} the original set $K$ appearing in the pointwise constraint \eqref{eq:nonlinear_constraint} simply has been replaced by its abstractly defined closed convex hull $K^{co}$. However, in order to give a meaningful lower bound on $\norm{\bar{\rho}(t,\cdot)}_{H^{-1}(\T)}$ a precise description of this convex hull will be crucial. 

Before turning to the hull though, we like to state a further simplification saying that it is enough to consider velocity fields $v$ for which the first component of the vertical average is vanishing. In other words, for an efficient mixing in the one-dimensional case there should be no nontrivial background velocity that moves the fluid in $x_1$-direction as a bulk without mixing it.

\begin{lemma}\label{lem:averaged_zero_velocity} Let $\rho_0$ be one-dimensional, $T>0$ and $(\rho,v)$ be a distributional solution of \eqref{eq:equations_distributional} satisfying \eqref{eq:condition_energy_weak_setting}, \eqref{eq:condition_rho_weak_setting}. Then there exists another distributional solution $(\eta,w)$ of \eqref{eq:equations_distributional} with $w$ satisfying \eqref{eq:condition_energy_weak_setting}, $\eta$ satisfying \eqref{eq:condition_rho_weak_setting} and such that $\norm{\rho(t,\cdot)}_{H^{-1}(\T^d)}=\norm{\eta(t,\cdot)}_{H^{-1}(\T^d)}$ for all $t\in[0,T)$. Moreover, $\bar{w}_1(t,x_1)=0$ for almost every $(t,x_1)\in(0,T)\times\T$.
\end{lemma} 
\begin{proof}
By Lemma \ref{lem:properties_of_averaged_solutions} \ref{itm:averaged_ii} the first component of the averaged velocity $\bar{v}$ is spatially constant almost everywhere for almost every $t$. Define $\Phi_t:\T^d\rightarrow\T^d$ by 
\[
\Phi_{t,1}(x):=x_1+\int_0^t\bar{v}_1(s)\:ds,\quad \Phi_{t,j}(x)=x_j,~j=2,\ldots,d,
\]
such that $\Phi_t$ is in fact the Lagrangian flow associated to the vectorfield $(t,x)\mapsto(\bar{v}_1(t),0,\ldots,0)^T=\bar{v}_1(t)e_1\in\R^d$. Note that $\Phi_t$ is a volume preserving diffeomorphism.

Furthermore, we define 
\begin{align*}
\eta(t,x):=\rho(t,\Phi_t(x)),\quad w(t,x):=v(t,\Phi_t(x))-\bar{v}_1(t)e_1,
\end{align*}
where $\rho$ is the solution associated with $v$. Clearly the field $w$ is also incompressible. It is also easy to check that $\eta\in L^\infty((0,T)\times\T^d)$ satisfies $\partial_t\eta+\divv(\eta w)=0$, $\eta(0,\cdot)=\rho_0$ in the sense of distributions.

Moreover, if $\Delta \varphi(t,\cdot)=\rho(t,\cdot)$, then $\Delta (\varphi(t,\cdot)\circ \Phi_t)=\eta(t,\cdot)$ and $\nabla(\varphi(t,\cdot)\circ\Phi_t)=\nabla\varphi(t,\cdot)\circ \Phi_t$. Thus the volume preserving nature of $\Phi_t$ implies $\norm{\eta(t,\cdot)}_{H^{-1}(\T^d)}=\norm{\rho(t,\cdot)}_{H^{-1}(\T^d)}$, $t\in[0,T)$. 

Finally, we compute
\begin{align*}
\bar{w}_1(t,x_1)=L^{-(d-1)}\int_{\T^{d-1}}v_1(t,\Phi_t(x))\:dx-\bar{v}_1(t)=\bar{v}_1(t)-\bar{v}_1(t)=0,
\end{align*}
and
\begin{align*}
\int_{\T^d}\abs{w(t,x)}^2\:dx&=\int_{\T^d}\abs{v(t,\Phi_t(x))-\bar{v}_1(t)e_1}^2\:dx\\
&=\int_{\T^d}\abs{v(t,x)}^2\:dx-2\bar{v}_1(t)\int_{\T^d}v_1(t,x)\:dx +L^d\bar{v}_1(t)^2\\
&=\int_{\T^d}\abs{v(t,x)}^2\:dx-L^d\bar{v}_1(t)^2\leq E.
\end{align*}
Thus $w$ satisfies all the requested properties.
\end{proof}

Next we state a consequence of being in the closed convex hull $K^{co}$, $K$ was defined in \eqref{eq:definition_of_K}, and having a vanishing first velocity component. 
\begin{lemma}\label{lem:convex_hull_with_zero_velocity}
Let $z=(\rho,v,m,e)$ be a point in $K^{co}$ with $v_1=0$. Then there holds
\begin{align}\label{eq:special_hull_inequality}
m_1^2\leq e(1-\rho^2).
\end{align}
\end{lemma}
The proof is relying on the explicit characterization of $K^{co}$, which has nothing to do with the vertical averages considered in this section and is therefore moved to Appendix \ref{sec:appendix_convex_hull}.

\subsection{General one-dimensional subsolutions}\label{sec:general_1D_subsols}

Motivated by Lemmas \ref{lem:properties_of_averaged_solutions}, \ref{lem:averaged_zero_velocity}, \ref{lem:convex_hull_with_zero_velocity} we collect the important properties of vertical averages in the following definition of a general one-dimensional averaged solution, or one-dimensional subsolution in the convex-integration jargon.
\begin{definition}\label{def:one_dimensional_subsolution}
Let $\rho_0$ be as in Definition \ref{def:one_dimensional_initial_data} and $T>0$. The pair $(\rho,m_1)\in L^{\infty}((0,T)\times\T)\times L^\infty(0,T;L^2(\T))$ is called a one-dimensional subsolution for $\rho_0$ provided 
\begin{enumerate}[(i)]
\item\label{itm:one_dim_subsol_i} $\norm{\rho}_{L^\infty((0,T)\times\T)}\leq 1$,
\item\label{itm:one_dim_subsol_ii} there holds $\partial_t\rho +\partial_{x_1}m_1=0$ on $(0,T)\times\T$, $\rho(0,\cdot)=\rho_0$ in the sense of distributions,
\item\label{itm:one_dim_subsol_iii} there exists $e\in L^\infty(0,T; L^1(\T))$ such that for almost every $t\in(0,T)$ there holds $L^{d-1}\int_{\T}e(t,x_1)\:dx_1\leq E$, 
\item\label{itm:one_dim_subsol_iv} and such that almost everywhere on $(0,T)\times\T$ there holds
$m_1^2\leq e(1-\rho^2)$.
\end{enumerate}
If in addition $\norm{\rho(t,\cdot)}_{H^{-1}(\T)}\rightarrow 0$ as $t\rightarrow T$ we say that $(\rho,m)$ is a perfectly mixing subsolution for $\rho_0$ with mixing time $T$.
\end{definition}
Note here that in view of Lemma \ref{lem:mixing_norm_Lipschitz} we again assume that the $H^{-1}$-norm of any $\rho$ as in Definition \ref{def:one_dimensional_subsolution} is well-defined for any $t\in[0,T)$.
\begin{definition}\label{def:optimal_mixing_time_for_1D_subsols} Let $\rho_0$ be a one-dimensional initial data.
The optimal mixing time for $\rho_0$ in the class of one-dimensional subsolutions $T^{1sub}_{mix}(\rho_0)$ is defined as the infimum of all $T$ for which there exists a subsolution in the sense of Definition \ref{def:one_dimensional_subsolution} that is perfectly mixing $\rho_0$ at time $T$.  
\end{definition}
Utilizing actual vertical averaging as introduced in Section \ref{sec:vertical_averages} and Lemmas \ref{lem:properties_of_averaged_solutions}, \ref{lem:averaged_zero_velocity}, \ref{lem:convex_hull_with_zero_velocity} one deduces the following relation.
\begin{lemma}\label{lem:existence_of_general_1D_subsolution} 
For one-dimensional $\rho_0$ and any distributional solution $(\rho,v)$ as in Theorem \ref{thm:weak_setting} there exists a one-dimensional subsolution $(\tilde{\rho},m_1)$ for $\rho_0$ with 
\begin{align*}
\norm{\rho(t,\cdot)}_{H^{-1}(\T^d)}\geq L^{\frac{d-1}{2}}\norm{\tilde{\rho}(t,\cdot)}_{H^{-1}(\T)},~t\in[0,T).
\end{align*}
In particular $T_{mix}(\rho_0)\geq T_{mix}^{1sub}(\rho_0)$ for all one-dimensional initial distributions $\rho_0$.
\end{lemma}

\subsection{Estimating the decay rate: Steepest descent}\label{sec:stepest_descent}

In order to estimate $T_{mix}^{1sub}(\rho_0)$ we begin by extending the computation leading to the bound \eqref{eq:first_lower_bound} to the level of one-dimensional subsolutions.

\begin{lemma}\label{lem:first_estimate_for_subsolutions}
Let $\rho_0$ be as in Definition \ref{def:one_dimensional_initial_data} and $(\rho,m_1)$ be a corresponding one-dimensional subsolution. Then for almost every $t\in(0,T)$ there holds 
\begin{align}\label{eq:derivative_mixing_norm_1D_subsols}
\frac{d}{dt}\norm{\rho(t,\cdot)}_{H^{-1}(\T)}^2\geq -2\left(\frac{E}{L^{d-1}}\right)^{\frac{1}{2}}\left(\int_{\T}\left(1-\rho(t,x_1)^2\right)(\partial_{x_1}\varphi(t,x_1))^2\:dx_1\right)^{\frac{1}{2}},
\end{align}
where $\varphi$ is the one-dimensional potential associated with $\rho$, i.e. $\partial_{x_1}^2\varphi(t,\cdot)=\rho(t,\cdot)$ for all $t\in[0,T)$.
\end{lemma}
\begin{proof}
By Lemma \ref{lem:mixing_norm_Lipschitz} and Definition \ref{def:one_dimensional_subsolution} \ref{itm:one_dim_subsol_ii} the left-hand side indeed exists at almost every $t\in(0,T)$ and is given by 
\begin{align*}
\frac{d}{dt}\norm{\rho(t,\cdot)}_{H^{-1}(\T)}^2=-2\int_{\T}m_1(t,x_1)\partial_{x_1}\varphi(t,x_1)\:dx_1.
\end{align*}
By Definition \ref{def:one_dimensional_subsolution} \ref{itm:one_dim_subsol_iii},\ref{itm:one_dim_subsol_iv} we then estimate for $h>0$
\begin{align*}
\int_{t}^{t+h}\int_{\T}&m_1\partial_{x_1}\varphi\:dx_1\:ds\leq \int_{t}^{t+h}\int_{\T}\sqrt{e(1-\rho^2)}\abs{\partial_{x_1}\varphi}\:dx_1\:ds\\
&\leq\int_t^{t+h}\left(\int_\T e\:dx_1\right)^{\frac{1}{2}}\left(\int_\T (1-\rho^2)(\partial_{x_1}\varphi)^2\:dx_1\right)^{\frac{1}{2}}\:ds\\
&\leq \left(\frac{E}{L^{d-1}}\right)^{\frac{1}{2}}\int_t^{t+h}\left(\int_\T (1-\rho^2)(\partial_{x_1}\varphi)^2\:dx_1\right)^{\frac{1}{2}}\:ds.
\end{align*}
Dividing by $h$ and passing to $h\rightarrow 0$, which is possible at almost every $t\in(0,T)$ since $(1-\rho^2)(\partial_{x_1}\varphi)^2\in L^\infty(0,T;L^1(\T))$, one arrives at \eqref{eq:derivative_mixing_norm_1D_subsols}.
\end{proof}

\begin{remark}\label{rem:sharpness_first_bound}
As the proof shows the above estimate is sharp for
\begin{align}\label{eq:sharp_choice_of_m_1D}
m_1(t,x_1)=\lambda(t)\left(1-\rho(t,x_1)^2\right)\partial_{x_1}\varphi(t,x_1),\\
e(t,x_1)=\lambda(t)^2\left(1-\rho(t,x_1)^2\right)\partial_{x_1}\varphi(t,x_1)^2,\label{eq:sharp_choice_of_e}
\end{align}
where $\lambda(t)$ is (formally when $\abs{\rho(t,\cdot)}=1$ or $\partial_{x_1}\varphi(t,\cdot)=0$ almost everywhere) defined as 
\begin{equation}\label{eq:lambda_of_t}
\lambda(t):=\left(\frac{E}{L^{d-1}}\right)^{\frac{1}{2}}\left(\int_{\T}\left(1-\rho(t,x_1)^2\right)(\partial_{x_1}\varphi(t,x_1))^2\:dx_1\right)^{-\frac{1}{2}}.
\end{equation}
Plugging this $m_1$ into the linear equation in Definition \ref{def:one_dimensional_subsolution} \ref{itm:one_dim_subsol_ii} one obtains that the in the sense of Lemma \ref{lem:first_estimate_for_subsolutions} steepest descending one-dimensional subsolution is characterized through the one-dimensional nonlocal conservation law
\begin{equation}\label{eq:1D_nonlocal_conservation_law}
\partial_t\rho+\lambda(t)\partial_{x_1}\left((1-\rho^2)\partial_{x_1}\varphi\right)=0,\quad \partial_{x_1}^2\varphi=\rho.
\end{equation}
This equation can be seen as the analogue of the instantaneously optimal stirring velocity field discussed in \cite[Section 4]{Lin_Thiffeault_Doering_2011}, but on the level of one-dimensional subsolutions instead of actual solutions. 

We also like to mention that equation \eqref{eq:1D_nonlocal_conservation_law}, at least if $\lambda(t)\equiv 1$, which in a first step can be achieved by rescaling time, fits very well into the class of nonlocal conservation laws considered by Amorim \cite{Amorim_2012}, who proves existence and $L^1$-stability estimates of entropy solutions emanating from BV initial data. 
\end{remark}

\subsection{Estimating the decay rate: Variational problems}\label{sec:variational_problem}

We continue to estimate the right-hand side of \eqref{eq:derivative_mixing_norm_1D_subsols}. In the following let us for simplicity write $x$ instead of $x_1$ and denote the spatial derivative $\partial_x w$ by $w_x$. For $h>0$ we introduce the set of functions 
\begin{equation}\label{eq:definition_of_X}
X_h:=\set{w\in \Lip(\T):\norm{w_x}_{L^\infty(\T)}\leq 1,~\int_{\T}w\:dx=0,~\int_{\T}w^2\:dx=h^2},
\end{equation}
as well as the functional $F:H^1(\T)\rightarrow \R$,
\begin{equation}\label{eq:definition_of_F}
F(w):=\int_{\T}(1-w_x^2)w^2\:dx.
\end{equation}
We remark that $F$ is well-defined, since $H^1(\T)$ embeds into the space of $\frac{1}{2}$-H\"older continuous, and thus bounded, functions.
Moreover, we will see in Lemma \ref{lem:estimate_L2_norm}  that the set $X_h$ is nonempty precisely for $h\in[0,L^\frac{3}{2}/\sqrt{48}]$.

Now for a one-dimensional subsolution $(\rho,m_1)$ with potential $\varphi$ we set $w:=\partial_{x_1}\varphi(t,\cdot)=\varphi_x(t,\cdot)$ to conclude the following statement from Lemma \ref{lem:first_estimate_for_subsolutions} and  Corollary \ref{cor:estimate_Hminus1_norm} below.
\begin{lemma}\label{lem:abstract_infimum}
Let $(\rho,m_1)$ be a one-dimensional subsolution and let us abbreviate $h(t):=\norm{\rho(t,\cdot)}_{H^{-1}(\T)}$. Then $h(t)^2\in[0,L^3/48]$, $t\in[0,T)$ and for almost every $t\in(0,T)$ there holds
\begin{align*}
\frac{d}{dt}h(t)^2\geq-2\left(\frac{E}{L^{d-1}}\right)^{\frac{1}{2}}\left(\sup_{w\in X_{h(t)}}F(w)\right)^{\frac{1}{2}}.
\end{align*}
\end{lemma}
It thus remains to investigate the stated supremum for which we will show the following statement.
\begin{lemma}\label{lem:supremum}
For $h\in [0,L^{\frac{3}{2}}/\sqrt{48}]$ the supremum of $F$ over $X_h$ is given by
\begin{align*}
\sup_{w\in X_h}F(w)=h^2-\frac{L^3}{48}\left(1-\sqrt{1-\frac{h^2}{L^3/48}}\right)^2.
\end{align*}
\end{lemma}

\subsubsection{Symmetric decreasing rearrangement}
In order to estimate $F(w)$ for $w\in X_h$ we will rely on symmetries. 
First of all we recall, see \cite[Proposition 1.30]{Baernstein_book_2019}, that the symmetric decreasing rearrangement of a Lebesgue measurable function $f:\R\rightarrow[0,\infty)$ with 
\begin{equation}\label{eq:condition_for_symmetric_decreasing_rearrangement}
\lambda_f(t):=\abs{\set{x\in\R:f(x)>t}}<\infty\quad\text{for all }t>\essinf f
\end{equation}
is uniquely characterized as the everywhere defined function $f^\sharp:\R\rightarrow\R$ that is equidistributed with $f$, i.e. $\lambda_f(t)=\lambda_{f^\sharp}(t)$ for all $t\in\R$, as well as even, non-increasing on $[0,\infty)$ and right continuous on $(0,\infty)$. The main properties of $f^\sharp$ that we will use are
\begin{gather}
\label{eq:symm_rearr_Lp_norms}
\int_{\R}f^p\:dx=\int_{\R}(f^\sharp)^p\:dx,~p\in[1,\infty),
\end{gather}
and that it decreases the modulus of continuity, as well as Dirichlet type integrals. More precisely, let $f\in \Lip(\R)$, $f\geq 0$ satisfy \eqref{eq:condition_for_symmetric_decreasing_rearrangement}, then for $1\leq p\leq \infty$ there holds
\begin{equation}\label{eq:decrease_of_W1p_norms}
\norm{f^\sharp_x}_{L^p(\R)}\leq \norm{f_x}_{L^p(\R)},
\end{equation}
see \cite[Theorems 3.6, 3.7]{Baernstein_book_2019}. For finite $p\geq 1$ this also holds true for functions just in the corresponding Sobolev space, \cite[Theorem 3.20]{Baernstein_book_2019}, but we will deal with Lipschitz functions anyway. Moreover, since $f\geq 0$ there also holds
\begin{gather}
\label{eq:symm_rearr_quadrat} (f^\sharp)^2=(f^2)^\sharp.
\end{gather}

The next lemma states that we can always find a symmetric competitor. We define
\begin{align}\label{eq:definition_of_X_without_h}
X:=\set{w\in \Lip(\T):\norm{w_x}_{L^\infty(\T)}\leq 1,~\int_{\T}w\:dx=0}.
\end{align}
\begin{lemma}\label{lem:existence_of_symmetric_decreasing_competitor}
Let $w\in X$. Then there exists $\tilde{w}\in X$ such that $\tilde{w}$ when restricted to $(-L/2,L/2)$ is even and non-increasing on $[0,L/2)$ and such that 
\begin{gather}\label{eq:symm_rearr_conserves_L2}
\int_{\T}\tilde{w}^2\:dx=\int_{\T}w^2\:dx,\\
\label{eq:decreasing_through_symmetric_rearrangement}
\int_{\T}\tilde{w}^2\tilde{w}_x^2\:dx\leq \int_{\T}w^2w_x^2\:dx,\\
\label{eq:conservation_of_inf_sup_symm_rearr}
\inf \tilde{w}=\inf w,\quad \sup \tilde{w}=\sup w.
\end{gather}
\end{lemma}
\begin{proof}
Let us abbreviate $l:=L/2$. After a translation we can assume that $w(\pm l)=0$ such that the two functions $\hat{w}_\pm:\R\rightarrow [0,\infty)$,\begin{align*}
\hat{w}_\pm(x)=\begin{cases}
w^\pm(x),&x\in(-l,l),\\
0,&x\notin(-l,l)
\end{cases}
\end{align*}
are still Lipschitz continuous with Lipschitz constant $\leq 1$. Here $w^\pm\geq 0$ denote the positive and negative part of $w$. Since $\hat{w}_\pm$ satisfy \eqref{eq:condition_for_symmetric_decreasing_rearrangement}, their symmetric decreasing rearrangements $v_\pm:=(\hat{w}_\pm)^\sharp:\R\rightarrow [0,\infty)$ are well-defined. By \eqref{eq:decrease_of_W1p_norms} for $p=\infty$ they are also Lipschitz continuous with Lipschitz constant $\leq 1$.

Since $v_\pm$ and $w_\pm$ are equidistributed we observe that $a_1:= \abs{\supp v_+}/2$, $a_2:=\abs{\supp v_-}/2$ satisfy
\begin{align}\label{eq:relation_supports}
a_1+a_2\leq l.
\end{align}

We then define $\tilde{w}:\T\rightarrow\R$ as the $2l$-periodic extension of $v:[-l,l]\rightarrow\R$,
\begin{align*}
v(x)=v_+(x)-v_-(x-l)-v_-(x+l).
\end{align*} 
As a consequence of \eqref{eq:relation_supports} and the monotonicity properties of $v_\pm$ there holds
\begin{align*}
v(x)=\begin{cases}
-v_-(x+l),&x\in[-l,-l+a_2),\\
v_+(x),&x\in(-a_1,a_1),\\
-v_-(x-l),&x\in (l-a_2,l],\\
0,&\text{otherwise}.
\end{cases}
\end{align*}
In particular $v(l)=v(-l)=-v_-(0)$, which implies that $\tilde{w}$ is well-defined as a Lipschitz continuous function with $\norm{\tilde{w}_x}_{L^\infty(\T)}\leq 1$.

Next we will verify the three requested integral (in)equalities $\int_{\T}\tilde{w}\:dx=0$, \eqref{eq:symm_rearr_conserves_L2} and \eqref{eq:decreasing_through_symmetric_rearrangement}. Actually we will do it only for \eqref{eq:decreasing_through_symmetric_rearrangement}, the other two can be seen in the same way by use of \eqref{eq:symm_rearr_Lp_norms}.
By translation of the two integrals where $v(x)=-v_-(x\pm l)$ there holds
\begin{align*}
\int_{\T}\tilde{w}^2\tilde{w}_x^2\:dx&=\int_{-a_1}^{a_1}v_+^2v_{+,x}^2\:dx+\int_{-a_2}^{a_2}v_-^2v_{-,x}^2\:dx\\
&=\int_{\R}\left(\left(\frac{1}{2}\left((\hat{w}_+)^\sharp\right)^2\right)_x\right)^2\:dx+\int_{\R}\left(\left(\frac{1}{2}\left((\hat{w}_-)^\sharp\right)^2\right)_x\right)^2\:dx.
\end{align*}
Now swapping symmetrization with squaring, cf. \eqref{eq:symm_rearr_quadrat}, and \eqref{eq:decrease_of_W1p_norms} for $p=2$ imply
\begin{align*}
\int_{\T}\tilde{w}^2\tilde{w}_x^2\:dx&\leq \int_{\R}\left(\left(\frac{1}{2}\left(\hat{w}_+\right)^2\right)_x\right)^2\:dx+\int_{\R}\left(\left(\frac{1}{2}\left(\hat{w}_-\right)^2\right)_x\right)^2\:dx\\
&=\int_{-l}^l(w^+)^2(w^+_x)^2+(w^-)^2(w^-_x)^2\:dx=\int_{\T}w^2w_x^2\:dx.
\end{align*}
Finally it is also easy to check that $v$ is even with respect to the origin and non-increasing on $[0,l]$.
\end{proof}
\subsubsection{Odd rearrangement}

In addition to the symmetric decreasing rearrangement $f^\sharp$ we will also use the odd rearrangement of a strictly increasing function introduced by Cabr\'e et al. \cite{Cabre_etal_antisymmetry_2018}. Again we will state its definition and its properties taylored to our needs and not in full generality. The interested reader should consult \cite{Cabre_etal_antisymmetry_2018} for more details.
\begin{definition}\label{def:odd_rearrangement}
Let $a>0$, $f\in \Lip([-a,a])$ be such that $\essinf f_x>0$ and $f(-a)=-f(a)$. Its flipped $f_*\in \Lip([-a,a])$ is defined as $f_*(x):=-f(-x)$. The odd rearrangement $f^o:[-a,a]\rightarrow\R$ of $f$ is defined as the inverse function of $\frac{1}{2}f^{-1}+\frac{1}{2}\left(f^{-1}\right)_*$, where $f^{-1}$ is the inverse of $f$.
\end{definition}
\begin{lemma}\label{lem:odd_rearr} For $f$ as in Definition \ref{def:odd_rearrangement} the odd rearrangement $f^o$ has the following properties:
\begin{enumerate}[(i)]
\item \label{itm:odd_rearr_i} $f^o$ is odd and $f^o(a)=f(a)$,
\item \label{itm:odd_rearr_ii} $f^o\in \Lip([-a,a])$ with $\essinf f_x\leq \essinf f^o_x\leq \esssup f^o_x\leq \esssup f_x$,
\item \label{itm:odd_rearr_iii} for every even continuous function $G:[-f(a),f(a)]\rightarrow\R$ there holds 
\[
\int_{-a}^aG\circ f^o\:dx=\int_{-a}^aG\circ f\:dx,
\]
\item \label{itm:odd_rearr_iv} for any even continuous $G:[-f(a),f(a)]\rightarrow [0,\infty)$ there holds
\[
\int_{-a}^a(f^o_x)^2G\circ f^o\:dx\leq \int_{-a}^a(f_x)^2G\circ f\:dx,
\]
and if $G$ is strictly positive almost everywhere, then equality holds if and only if $f=f_*=f^o$.
\end{enumerate}
\end{lemma}
\begin{proof}
First of all note that the condition $f(-a)=-f(a)$ implies that the map $g:=\frac{1}{2}f^{-1}+\frac{1}{2}(f^{-1})_*$ is well-defined from $[-f(a),f(a)]$ to $[-a,a]$. It is easy to see that $g$ is odd and $g(f(a))=a$. Thus $f^o=g^{-1}$ satisfies \ref{itm:odd_rearr_i}. Also \ref{itm:odd_rearr_ii} follows in a straightforward way by the properties of Lipschitz maps and the fact that $g$ is a convex combination of $f^{-1}$ and $(f^{-1})_*$. Property \ref{itm:odd_rearr_iii} can be seen via change of variables.

Also in \ref{itm:odd_rearr_iv} we transform the integrals to rewrite
\begin{align*}
\int_{-a}^a (f^o_x)^2G\circ f^o\:dx=\int_{-f(a)}^{f(a)}\frac{G(y)}{g'(y)}\:dy=\int_{-f(a)}^{f(a)}\frac{G(y)}{\frac{1}{2}(f^{-1})'(y)+\frac{1}{2}((f^{-1})_*)'(y)}\:dy,
\end{align*}
as well as
\begin{align*}
\int_{-a}^a f_x^2G\circ f\:dx&=\frac{1}{2}\int_{-a}^af_x^2G\circ f\:dx+\frac{1}{2}\int_{-a}^a(f_{*,x})^2G\circ f_*\:dx\\
&=\frac{1}{2}\int_{-f(a)}^{f(a)}\frac{G(y)}{(f^{-1})'(y)}\:dy+\frac{1}{2}\int_{-f(a)}^{f(a)}\frac{G(y)}{((f_*)^{-1})'(y)}\:dy.
\end{align*}
Since $(f_*)^{-1}=(f^{-1})_*$ the stated inequality follows by pointwise comparison and the convexity of $(0,\infty)\ni z\mapsto 1/z\in\R$.
 
Moreover, if $G$ is strictly positive almost everywhere, by using the strict convexity of $1/z$ one sees that equality implies $(f^{-1})'(y)=((f_*)^{-1})'(y)$ for almost every $y\in (-f(a),f(a))$. Due to $f(a)=f_*(a)$ it follows $f^{-1}=(f_*)^{-1}$ and thus $f=f_*=f^o$.
\end{proof}

In order to apply the odd rearrangement we introduce the following two sets of functions
\begin{align}\label{eq:definition_W}
\cW&:=\Big\{w\in \Lip([-L/4,L/4]):0\leq w_x\leq 1\text{ a.e.,~}\int_{-L/4}^{L/4}w\:dx=0\Big\},\\\label{eq:definition_W0}
\cW_0&:=\Big\{w\in \Lip([-L/4,L/4]):0\leq w_x\leq 1\text{ a.e.,~}w(0)=0\Big\}.
\end{align}
\begin{lemma}\label{lem:application_of_odd_rearrangement}
For each $w\in\cW\setminus\cW_0$ there exists $\tilde{w}\in \cW_0$ with
\begin{align}\label{eq:odd_rearr_conserves_L2}
\int_{-L/4}^{L/4}\tilde{w}^2\:dx=\int_{-L/4}^{L/4}w^2\:dx,\\
\label{eq:odd_rearr_decreases_F}
\int_{-L/4}^{L/4}\tilde{w}^2\tilde{w}_x^2\:dx\leq \int_{-L/4}^{L/4}w^2w_x^2\:dx.
\end{align}
\end{lemma}
We remark that $\tilde{w}$ might have lost the vanishing average property, i.e. potentially $\tilde{w}\notin \cW$. 

\begin{proof}[Proof of Lemma \ref{lem:application_of_odd_rearrangement}] We abbreviate again $l:=L/2$.
Let $w\in \cW\setminus \cW_0$, i.e. $w(0)\neq 0$. We claim that there exists a point $x_0\in (0,l/2]$ such that $w(-x_0)=-w(x_0)$. Otherwise $w(x)+w(-x)$ has a distinguished sign on $(0,l/2]$, say $w(x)+w(-x)>0$ for all $x\in(0,l/2]$. But then integration gives the contradiction
\[
0<\int_0^{l/2}w(x)+w(-x)\:dx=\int_{-l/2}^{l/2}w(x)\:dx=0.
\]
Now take $x_0\in(0,l/2]$ as above and define for $\varepsilon>0$ the function $w_\varepsilon(x):=w(x)+\varepsilon x$ which now has a uniformly positive slope and also satisfies $w_\varepsilon(-x_0)=-w_\varepsilon(x_0)$. Thus we can apply the odd rearrangement as stated in Definition \ref{def:odd_rearrangement} to the restriction $w_{\varepsilon|[-x_0,x_0]}$. We then set $\tilde{w}_\varepsilon:[-l/2,l/2]\rightarrow \R$,
\begin{align}\label{eq:definitino_wtilde_epsilon}
\tilde{w}_\varepsilon(x)=\begin{cases}
\left(w_{\varepsilon|[-x_0,x_0]}\right)^o(x),&x\in(-x_0,x_0),\\
w_\varepsilon(x),&\text{else}.
\end{cases}
\end{align}
Properties \ref{itm:odd_rearr_i},\ref{itm:odd_rearr_ii} of Lemma \ref{lem:odd_rearr} show that $\tilde{w}_\varepsilon\in \Lip([-l/2,l/2])$ with $0\leq \tilde{w}_{\varepsilon,x}\leq 1+\varepsilon$ almost everywhere. We also know $\tilde{w}_\varepsilon(0)=0$. Along a uniformly converging subsequence for $\varepsilon\rightarrow 0$ we therefore find a Lipschitz continuous limit $\tilde{w}\in \Lip([-l/2,l/2])$ with $0\leq \tilde{w}_x\leq 1$ and $\tilde{w}(0)=0$, i.e. $\tilde{w}\in\cW_0$.

Next, by construction $\tilde{w}$ coincides with $w$ outside of $(-x_0,x_0)$. Thus Lemma \ref{lem:odd_rearr} \ref{itm:odd_rearr_iii} implies
\begin{align*}
\int_{-l/2}^{l/2}\tilde{w}^2\:dx-\int_{-l/2}^{l/2}w^2\:dx&=\lim_{\varepsilon\rightarrow 0}\int_{-x_0}^{x_0}\big((w_{\varepsilon|[-x_0,x_0]})^o\big)^2-w^2\:dx\\
&=\lim_{\varepsilon\rightarrow 0}\int_{-x_0}^{x_0}w_\varepsilon^2-w^2\:dx=0.
\end{align*}
This shows property \eqref{eq:odd_rearr_conserves_L2}.

We recall that $\tilde{w}$ is the uniform limit of Lipschitz functions $\tilde{w}_\varepsilon$ with $0\leq \tilde{w}_{\varepsilon,x} \leq 1+\varepsilon$. In addition we can assume that $(\tilde{w}_\varepsilon^2)_x$ converges weakly in $L^2((-l/2,l/2))$ to $(\tilde{w}^2)_x$. By Lemma \ref{lem:odd_rearr} \ref{itm:odd_rearr_iv} with $G(y)=y^2$ we conclude
\begin{align*}
\int_{-l/2}^{l/2}\tilde{w}^2\tilde{w}_x^2\:dx&=\int_{-l/2}^{l/2}\left(\frac{1}{2}(\tilde{w}^2)_x\right)^2\:dx\leq \liminf_{\varepsilon\rightarrow 0}\int_{-l/2}^{l/2}\left(\frac{1}{2}(\tilde{w}_\varepsilon^2)_x\right)^2\:dx\\
&\leq \liminf_{\varepsilon\rightarrow 0}\int_{-l/2}^{l/2}\left(\frac{1}{2}(w_\varepsilon^2)_x\right)^2\:dx=\int_{-l/2}^{l/2}w^2w_x^2\:dx.
\end{align*}
This shows property \eqref{eq:odd_rearr_decreases_F} and finishes the proof of the lemma.
\end{proof}

For later use let us summarize how $\tilde{w}$ is formed.
\begin{lemma}\label{lem:application_of_odd_rearr2}
The function $\tilde{w}\in\cW_0$ from Lemma \ref{lem:application_of_odd_rearrangement} is the uniform limit of Lipschitz functions $\tilde{w}_\varepsilon$ defined through \eqref{eq:definitino_wtilde_epsilon} where $x_0$ is some point in $(0,L/4]$ and $w_\varepsilon(x)=w(x)+\varepsilon x$.
\end{lemma}
\subsubsection{Problem 1}

Having introduced the two rearrangements we first deal with the well-definedness of the sets $X_h$. Let $w_0\in \Lip(\T)$ be defined as the $L$-periodic extension of 
\[
w_0(x)=L/4-\abs{x},\quad x\in[-L/2,L/2]
\]
and recall the definition of $X$ in \eqref{eq:definition_of_X_without_h} if needed.
\begin{lemma}\label{lem:estimate_L2_norm}
There holds 
\begin{align*}
\sup_{w\in X}\int_\T w^2\:dx=\frac{L^3}{48}.
\end{align*}
The supremum is attained precisely for translates of $w_0$.
\end{lemma}
This in particular implies the following bound on the $H^{-1}$-norm:
\begin{corollary}\label{cor:estimate_Hminus1_norm} For any $\rho\in L^\infty(\T)$ with $\norm{\rho}_{L^\infty(\T)}\leq 1$ and $\int_\T\rho\:dx=0$ there holds $\norm{\rho}^2_{H^{-1}(\T)}\leq L^3/48$.
\end{corollary}

\begin{proof}[Proof of Lemma \ref{lem:estimate_L2_norm}] Let us abbreviate the stated supremum by $S$ and define $l:=L/2$.

\emph{First reduction:} By means of Lemma \ref{lem:existence_of_symmetric_decreasing_competitor}, in particular \eqref{eq:symm_rearr_conserves_L2}, we can without loss of generality restrict ourselves to functions $w\in X$ that are even when restricted to $(-l,l)$ and non-increasing on $[0,l)$.
After a shift by $l/2$ we therefore can characterize $S$ by
\begin{align}\label{eq:supremum1}
S=2\sup\set{\int_{-l/2}^{l/2}w^2\:dx:w\in \cW},
\end{align}
where we recall that the set $\cW$ defined in \eqref{eq:definition_W} consists of all Lipschitz functions $w:[-l/2,l/2]\rightarrow\R$ with $0\leq w_x\leq 1$ almost everywhere and $\int_{-l/2}^{l/2}w\:dx=0$.

\emph{Second reduction:} Next we apply Lemma \ref{lem:application_of_odd_rearrangement}, in particular \eqref{eq:odd_rearr_conserves_L2}, to deduce that
\begin{equation}\label{eq:supremum2}
\sup\set{\int_{-l/2}^{l/2}w^2\:dx:w\in \cW}\leq \sup\set{\int_{-l/2}^{l/2}w^2\:dx:w\in \cW_0},
\end{equation}
where $\cW_0$, defined in \eqref{eq:definition_W0}, is the set of Lipschitz functions $w\in\Lip([-l/2,l/2])$ with $0\leq w_x\leq 1$ almost everywhere and $w(0)=0$. Note that there holds only an inequality, since the set $\cW_0$ is not contained in $\cW$.

\emph{The reduced problem:}
We claim that the supremum over $\cW_0$ is achieved precisely for $w(x)=x$. For that let $w\in\cW_0$ and observe that then $0\leq w(x)\leq x$ for all $x\in [0,l/2]$ and $x\leq w(x)\leq 0$ for all $x\in [-l/2,0]$. Therefore, if $w$ is different from the identity strict convexity indeed implies
\begin{align*}
\int_{-l/2}^{l/2}x^2\:dx-\int_{-l/2}^{l/2}w^2\:dx>2\int_{-l/2}^{l/2}w(x)(x-w(x))\:dx\geq 0.
\end{align*}

\emph{Conclusion and uniqueness:} Observe that the identity $x\mapsto x$ belongs to both sets $\cW$ and $\cW_0$. Since the supremum over $\cW_0$ is attained if $w$ is the identity, we conclude equality in \eqref{eq:supremum2}, and in combination with \eqref{eq:supremum1} we deduce
\begin{align*}
S=2\int_{-l/2}^{l/2}x^2\:dx=\frac{l^3}{6}=\frac{L^3}{48}.
\end{align*}

Note that considering the identity in $\cW$ corresponds to considering a translate of $w_0$, defined in the beginning of this subsection, in the original variational problem. Thus the supremum $S$ is attained in $w_0$ and its translates.

Let us also show that the supremum $S$ is only achieved for such translates.
Indeed, Lemma \ref{lem:existence_of_symmetric_decreasing_competitor} applied to $u\in\Lip(\T)$ with $\int_\T u\:dx=0$, $\norm{u_x}_{L^\infty(\T)}\leq 1$, $\int_{\T}u^2\:dx=S=L^3/48$ gives an even competitor $\tilde{u}$ that is increasing on $(-l,0)$ and decreasing on $(0,l)$. Let us denote the shifted version which is increasing on $(-l/2,l/2)$ by $v$, i.e. $v_{|(-l/2,l/2)}\in\cW$ with 
\[
\int_{-l/2}^{l/2}v^2\:dx=\frac{S}{2}=\sup_{w\in\cW_0}\int_{-l/2}^{l/2}w^2\:dx.
\]
If $v(0)=0$, i.e. $v_{|(-l/2,l/2)}\in \cW_0$ we conclude $v(x)=x$, $x\in(-l/2,l/2)$ by the uniqueness of the reduced problem. If $v(0)\neq 0$, by Lemma \ref{lem:application_of_odd_rearr2} we find a sequence of functions $\tilde{v}_\varepsilon$ which uniformly converges to the identity by means of the same uniqueness result. In addition we can also assume that $\tilde{v}_{\varepsilon,x}$ converges weakly in $L^2(-l/2,l/2)$ to $1$. Now recall from Lemma \ref{lem:application_of_odd_rearr2} that outside a certain interval $(-x_0,x_0)$, which is independent of $\varepsilon>0$, there holds $\tilde{v}_\varepsilon(x)=v(x)+\varepsilon x$, and thus $v(x)=x$ for $x\notin(-x_0,x_0)$. Moreover, inside $(-x_0,x_0)$ where the odd rearrangement is applied we use Lemma \ref{lem:odd_rearr} \ref{itm:odd_rearr_iv} with $G(y)=1$ and the weak convergence of $\tilde{v}_{\varepsilon,x}$ to estimate
\begin{align*}
\int_{-x_0}^{x_0}1\:dx\leq \liminf_{\varepsilon\rightarrow 0}\int_{-x_0}^{x_0}(\tilde{v}_{\varepsilon,x})^2\:dx\leq \liminf_{\varepsilon\rightarrow 0}\int_{-x_0}^{x_0}(v_x+\varepsilon)^2\:dx=\int_{-x_0}^{x_0}(v_x)^2\:dx.
\end{align*}
On the otherhand we know $0\leq v_x\leq 1$ almost everywhere. Thus $v_x=1$ almost everywhere in $(-x_0,x_0)$. But then, since we assumed $v(0)\neq 0$, we get a contradiction to $\int_{-l/2}^{l/2}v\:dx=0$. Consequently $v(x)=x$, $x\in(-l/2,l/2)$.

It remains to show that this implies $u$ to coinced with $w_0$. The definition of $v$, the fact that $v(x)=x$ for $x\in(-l/2,l/2)$ and \eqref{eq:conservation_of_inf_sup_symm_rearr} of Lemma \ref{lem:existence_of_symmetric_decreasing_competitor} imply that the infimum of $u$  coincides with $\inf w_0=w_0(\pm l)=L/4-l=-L/4$, and that $\sup u=\sup w_0=L/4$. After a translation we can assume that also $u$ attains its infimum in $\pm l$. However, under the constraint $\norm{u_x}_{L^\infty(\T)}\leq 1$ this is only possible for $u=w_0$.
\end{proof}

\subsubsection{Problem 2}
In view of Lemma \ref{lem:estimate_L2_norm} we define 
\begin{align}\label{eq:defintino_of_hmax}
h_{max}:=\frac{L^\frac{3}{2}}{\sqrt{48}}
\end{align}
and fix $h\in [0,h_{max}]$. Recall that we are interested in the supremum of $F$ over the set $X_h$. The definitions of $F$ and $X_h$, cf. \eqref{eq:definition_of_X}, \eqref{eq:definition_of_F}, imply that 
\begin{align}\label{eq:relation_supF_and_infF1}
\sup_{w\in X_h}F(w)=h^2-\inf_{w\in X_h} F_1(w),
\end{align} 
where $F_1:H^1(\T)\rightarrow \R$ is defined as 
\begin{align*}
F_1(w):=\int_{\T}w_x^2w^2\:dx.
\end{align*}
In order to study $\inf_{X_h}F_1$ we proceed analogously to Lemma \ref{lem:estimate_L2_norm}. However, here we first of all deal with the corresponding reduced variational problem.

Again we abbreviate $l:=L/2$ and define the set $\cW_0^h$ by saying that $w\in\cW_0^h$ if and only if
\begin{gather*}
w\in \Lip([-l/2,l/2]),\quad 0\leq w_x\leq 1,\quad w(0)=0,\quad \int_{-l/2}^{l/2}w^2\:dx=\frac{h^2}{2}.
\end{gather*}
Moreover, for $\alpha\in[0,1]$ we define an odd function $w^\alpha:[-l/2,l/2]\rightarrow\R$ by setting 
\begin{align*}
w^\alpha(x):=\begin{cases}
x,&0\leq x\leq \alpha^2\frac{l}{2},\\
\alpha\frac{l}{2}\left(1-\frac{\left(1-\frac{x}{l/2}\right)^2}{1-\alpha^2}\right)^{\frac{1}{2}},&\alpha^2\frac{l}{2}<x\leq \frac{l}{2}.
\end{cases}
\end{align*}
We also set $F_2:H^1(-l/2,l/2)\rightarrow\R$, 
\[
F_2(w)=\int_{-l/2}^{l/2}w_x^2w^2\:dx.
\]
\begin{lemma}\label{lem:reduced_var_problem} The function $w^\alpha$ belongs to $\cW_0^h$ if and only if $\alpha=\alpha(h)$ is the unique number in $[0,1]$ which satisfies
\begin{equation}\label{eq:condition_on_alpha}
\frac{h^2}{h_{max}^2}=\alpha^2(2-\alpha^2).
\end{equation}
Moreover, for the infimum of $F_2$ over $\cW_0^h$ there holds
\begin{align*}
\inf_{w\in\cW_0^h}F_2(w)=F_2(w^{\alpha(h)})=\frac{h_{max}^2}{2}\left(1-\sqrt{1-\frac{h^2}{h_{max}^2}}\right)^2
\end{align*}
and $w^{\alpha(h)}$ is the only element of $\cW_0^h$ in which the infimum is attained.
\end{lemma}
\begin{proof}
Note that $w^\alpha$ is even continuously differentiable with 
\begin{align}\label{eq:derivative_of_walpha}
w^\alpha_x(x)=\alpha\left(1-\frac{\left(1-\frac{x}{l/2}\right)^2}{1-\alpha^2}\right)^{-\frac{1}{2}}\frac{1-\frac{x}{l/2}}{1-\alpha^2}
\end{align}
for $\alpha^2l/2< x\leq l/2$. Clearly $w^\alpha(0)=0$, and \eqref{eq:derivative_of_walpha} implies $0\leq w_x^\alpha\leq 1$.   

Concerning the $L^2$-integral we compute
\begin{align*}
2\int_{-l/2}^{l/2}(w^\alpha)^2\:dx=\frac{l^3}{6}\alpha^2(2-\alpha^2).
\end{align*}
Now $w^\alpha\in\cW_0^h$ if and only if the latter quantity coincides with $h^2$. Recalling \eqref{eq:defintino_of_hmax} this is indeed equivalent to \eqref{eq:condition_on_alpha}.

Next we consider $w\in\cW_0^h$ different from $w^{\alpha(h)}$. In view of the uniqueness in Lemma \ref{lem:estimate_L2_norm} this implies $h<h_{max}$ and thus $\alpha<1$. We define $u:=\frac{1}{2}w^2$ and $\bar{u}:=\frac{1}{2}(w^{\alpha(h)})^2$. Note that this implies that also $u_x$ and $\bar{u}_x$ are different from another. By strict convexity and since $\bar{u}_x(l/2)=0$, $\bar{u}_x(-l/2)=0$ we obtain
\begin{align*}
F_2(w)-F_2(w^{\alpha(h)})&=\int_{-l/2}^{l/2}u_x^2\:dx-\int_{-l/2}^{l/2}\bar{u}_x^2\:dx>2\int_{-l/2}^{l/2}\bar{u}_x(u_x-\bar{u}_x)\:dx\\
&=-2\int_{-l/2}^{l/2}\bar{u}_{xx}(u-\bar{u})\:dx.
\end{align*}
Now $\bar{u}_{xx}(x)=1$ for $x\in (-\alpha^2l/2,\alpha^2l/2)$, $\bar{u}_{xx}(x)=-\frac{\alpha^2}{1-\alpha^2}$ for $\abs{x}\in (\alpha^2 l/2,l/2)$ and $\int_{-l/2}^{l/2}(u-\bar{u})\:dx=0$, which holds due to the fact that $w,w^{\alpha(h)}\in\cW_0^h$, imply
\begin{align*}
F_2(w)-F_2(w^{\alpha(h)})&>-2\int_{-l/2}^{l/2}\bar{u}_{xx}(u-\bar{u})\:dx\\
&=-2\left(1+\frac{\alpha^2}{1-\alpha^2}\right)\int_{-\alpha^2l/2}^{\alpha^2l/2}(u-\bar{u})\:dx\\
&=\frac{1}{1-\alpha^2}\int_{-\alpha^2l/2}^{\alpha^2l/2}x^2-w^2\:dx.
\end{align*}
Since $w$ is $1$-Lipschitz with $w(0)=0$ we conclude $F_2(w)-F_2(w^{\alpha(h)})>0$.

Finally one computes
\begin{align*}
F_2(w^{\alpha(h)})=\frac{l^3}{12}\alpha(h)^4,
\end{align*}
which in view of \eqref{eq:defintino_of_hmax}, \eqref{eq:condition_on_alpha}  becomes
\[
\frac{l^3}{12}\alpha(h)^4=\frac{1}{2}h_{max}^2\left(1-\sqrt{1-\frac{h^2}{h_{max}^2}}\right)^2.
\]
This finishes the proof of the lemma.
\end{proof}
It remains to reduce the actual variational problem to the one just solved. 
For that we define $\bar{w}^h\in \Lip(\T)$ to be the $L$-periodic extension of the even extension $[-l,l]\rightarrow\R$ of $[-l,0]\rightarrow \R$, $x\mapsto w^{\alpha(h)}(x+l/2)$.
\begin{lemma}\label{lem:reduction_of_actual_variational_problem}
There holds
\begin{align*}
\inf_{w\in X_h}F_1(w)=2\inf_{w\in\cW_0^h}F_2(w)
\end{align*}
and the infimum on the left-hand side is achieved in translations of $\bar{w}^h$. 
\end{lemma}
\begin{proof} Let $w\in X_h$ and $l:=L/2$. 

\emph{First reduction:} By Lemma \ref{lem:existence_of_symmetric_decreasing_competitor} there exists $\tilde{w}\in X_h$ even on $(-l,l)$ and non-increasing on $(0,l)$ with $F_1(\tilde{w})\leq F_1(w)$.
Hence, after a shift we have that the function $\tilde{w}(\cdot -l/2)_{|[-l/2,l/2]}$ belongs to $\cW^h:=\cW\cap\set{w:\norm{w}_{L^2(-l/2,l/2)}^2=h^2/2}$ and thus
\begin{align*}
2\inf_{w\in\cW^h}F_2(w)\leq \inf_{w\in X_h}F_1(w).
\end{align*}
Note that the inequality in the other direction easily holds true by constructing even elements in $X_h$ out of the elements from $\cW^h$. Therefore
\begin{align*}
2\inf_{w\in\cW^h}F_2(w)=\inf_{w\in X_h}F_1(w).
\end{align*}

\emph{Second reduction:} Now given $w\in \cW^h\setminus \cW_0^h$ we find via Lemma \ref{lem:application_of_odd_rearrangement} a function $\tilde{w}\in \cW_0^h$ with $F_2(\tilde{w})\leq F_2(w)$. This gives
\begin{align*}
\inf_{w\in\cW_0^h}F_2(w)\leq \inf_{w\in \cW_h}F_2(w). 
\end{align*}

\emph{Conclusion:} However, by Lemma \ref{lem:reduced_var_problem} we know that $\inf_{\cW_0^h}F_2$ is achieved in $w^{\alpha(h)}$ which belongs to $\cW^h$ as well. In combination with the first step we conclude
\[
\inf_{w\in X_h}F_1(w)=2\inf_{w\in\cW^h}F_2(w)=2\inf_{w\in\cW_0^h}F_2(w).
\]
Finally, by definition of $\bar{w}^h$ and Lemma \ref{lem:reduced_var_problem} it is easy to check that the infimum of $F_1$ is achieved at $\bar{w}^h$.
\end{proof}

\begin{remark} Contrary to the first variational problem solved in Lemma \ref{lem:estimate_L2_norm} it here seems harder to trace back the uniqueness (up to translations) of the reduced problem through the two rearrangements. However, it is not needed for the rest of the analysis.
\end{remark}

\begin{proof}[Proof of Lemma \ref{lem:supremum}]
Relation \eqref{eq:relation_supF_and_infF1} and Lemmas \ref{lem:reduced_var_problem}, \ref{lem:reduction_of_actual_variational_problem} imply
\begin{align*}
\sup_{w\in X_h}F(w)=h^2-2\inf_{w\in\cW_0^h}F_2(w)=h^2-h_{max}^2\left(1-\sqrt{1-\frac{h^2}{h_{max}^2}}\right)^2
\end{align*}
with $h_{max}^2=L^3/48$.
\end{proof}

\subsection{Analysis of the derived ODE}\label{sec:analysis_of_ode}
Setting again $h_{max}:=L^{\frac{3}{2}}/\sqrt{48}$ and summarizing Lemmas \ref{lem:abstract_infimum}, \ref{lem:supremum}, as well as Corollary \ref{cor:estimate_Hminus1_norm} we have shown the following estimate:

\begin{proposition}\label{prop:ode_estimate}
Let $(\rho,m)$ be a one-dimensional subsolution. Define $q(t)$ through $\norm{\rho(t,\cdot)}_{H^{-1}(\T)}=q(t)h_{max}$. Then $q(t)\in[0,1]$ and  for almost every $t\in(0,T)$ there holds
\begin{align}\label{eq:ode_inequaltiy}
\frac{d}{dt}q(t)^2\geq-2\left(\frac{E}{L^{d-1}}\right)^{\frac{1}{2}}h_{max}^{-1}\left(q(t)^2-\left(1-\sqrt{1-q(t)^2}\right)^2\right)^{\frac{1}{2}}.
\end{align}
\end{proposition}

\begin{proposition}\label{prop:bound_for_1D_subs}
Let $\rho_0$ be one-dimensional and set $q_0:=\norm{\rho_0}_{H^{-1}(\T)}/h_{max}$. Then
\begin{align}\label{eq:bound_for_Tmix1}
T_{mix}^{1sub}(\rho_0)\geq h_{max}\left(\frac{L^{d-1}}{E}\right)^{\frac{1}{2}}\sqrt{2}S\left(\alpha_0\right)
\end{align}
where
\begin{equation}\label{eq:definition_of_S}
S(\alpha):=\frac{1}{2}\left(\arcsin(\alpha)+\alpha\sqrt{1-\alpha^2}\right),\quad \alpha_0:=\left(1-\sqrt{1-q_0^2}\right)^{\frac{1}{2}}.
\end{equation}
\end{proposition}
\begin{proof} Let $(\rho,m)$ be a one-dimensional subsolution that perfectly mixes the initial distribution $\rho_0$ at time $T$. Define $q(t)$ as in Proposition \ref{prop:ode_estimate} and set 
\begin{align*}
c:=2^{-\frac{1}{2}}E^{\frac{1}{2}}L^{-\frac{d-1}{2}}h_{max}^{-1}.
\end{align*}
Without loss of generality we will assume that $q(t)\in (0,1)$ for $t\in(0,T)$. 

In order to bound $T$ from below it is more convenient to introduce
\begin{align*}
\alpha(t):=\left(1-\sqrt{1-q(t)^2}\right)^{\frac{1}{2}},
\end{align*}
such that
\begin{align*}
q(t)^2=\alpha(t)^2\left(2-\alpha(t)^2\right)
\end{align*}
Note that $\alpha:[0,T]\rightarrow [0,1]$ is continuous and by our assumption $\alpha_{|(0,T)}$ is locally Lipschitz continuous with values in $(0,1)$. 

We will now write inequality \eqref{eq:ode_inequaltiy} in terms of $\alpha(t)$. First of all one computes
\begin{align*}
\dot{\alpha}(t)=\frac{1}{4\alpha(t)(1-\alpha(t)^2)}\frac{d}{dt}q(t)^2
\end{align*}
and \eqref{eq:ode_inequaltiy} becomes
\begin{align*}
\frac{d}{dt}q(t)^2\geq -2\left(\frac{E}{L^{d-1}}\right)^\frac{1}{2}h_{max}^{-1}\alpha(t)\sqrt{2(1-\alpha(t)^2)}
\end{align*}
Hence \eqref{eq:ode_inequaltiy} is equivalent to
\begin{align*}
(1-\alpha(t)^2)^{\frac{1}{2}}\dot{\alpha}(t)\geq -c
\end{align*}
for almost every $t\in(0,T)$.

Integration gives
\[
S(\alpha(t_2))-S(\alpha(t_1))\geq -c(t_2-t_1),\quad 0<t_1<t_2<T,
\]
with $S:[0,1]\rightarrow[0,\pi/4]$ defined in \eqref{eq:definition_of_S}.
Letting $t_1\rightarrow 0$, $t_2\rightarrow T$ we deduce
\begin{align*}
T\geq \frac{S(\alpha_0)}{c}=h_{max}\left(\frac{L^{d-1}}{E}\right)^{\frac{1}{2}}\sqrt{2}S\left(\alpha_0\right).
\end{align*}
This shows the stated lower bound for $T_{mix}^{1sub}$.
\end{proof}

\begin{proof}[Proof of Theorem  \ref{thm:lower_bound_1D} and Theorem \ref{thm:weak_setting}.]
Recall that the $H^{-1}$-norms on $\T$ and $\T^d$ for any one-dimensional data are related by $\norm{\rho_0}_{H^{-1}(\T^d)}=L^{\frac{d-1}{2}}\norm{\rho_0}_{H^{-1}(\T)}$, cf. Lemma \ref{lem:properties_of_averaged_solutions} \ref{itm:averaged_iv}. Hence Lemma \ref{lem:estimate_L2_norm} implies 
\[
\norm{\rho_0}_{H^{-1}(\T^d)}\leq h_{max}L^{\frac{d-1}{2}}=\frac{L^{\frac{d+2}{2}}}{\sqrt{48}}=H_{max}
\] for any one-dimensional $\rho_0$. This shows the first part of Theorem \ref{thm:lower_bound_1D}. The other part will be a direct consequence of Theorem \ref{thm:weak_setting}.

Let $(\rho,v)$ be a distributional solution as in Theorem \ref{thm:weak_setting} emanating from $\rho_0$. By Lemma \ref{lem:existence_of_general_1D_subsolution} there exists a one-dimensional subsolution $(\tilde{\rho},m_1)$ with
\begin{align*}
\norm{\rho(t,\cdot)}_{H^{-1}(\T^d)}\geq L^{\frac{d-1}{2}}\norm{\tilde{\rho}(t,\cdot)}_{H^{-1}(\T)}
\end{align*}
for all $t\in[0,T)$. This inequality and Proposition \ref{prop:bound_for_1D_subs} imply that $\norm{\rho(t,\cdot)}_{H^{-1}(\T^d)}$ is positive for times
\begin{align*}
t<h_{max}\left(\frac{L^{d-1}}{E}\right)^{\frac{1}{2}}\sqrt{2}S\left(\alpha_0\right)=H_{max}E^{-\frac{1}{2}}\sqrt{2}S(\alpha_0).
\end{align*}
This finishes the proof of Theorem \ref{thm:weak_setting}.
\end{proof}

\section{Sharpness}\label{sec:sharpness}
We will now prove Theorem \ref{thm:sharpness_weak_setting}. In other words we give an example for which Theorem \ref{thm:weak_setting} is sharp. Let $\hat{\rho}_0\in L^\infty(\T)$ be the unique one-dimensional ``maximally unmixed'' state
\begin{align}\label{eq:maximally_unmixed_data}
\hat{\rho}_0(x)=\begin{cases}
1,&x\in(-L/2,0),\\
-1,&x\in (0,L/2).
\end{cases}
\end{align}
By ``maximally unmixed'' we mean that for this configuration $\norm{\hat{\rho}_0}_{H^{-1}(\T)}=h_{max}=L^{\frac{3}{2}}/\sqrt{48}$, cf. Lemma \ref{lem:estimate_L2_norm}. Proposition \ref{prop:bound_for_1D_subs} therefore states that 
\begin{equation}\label{eq:lower_bound_for_rhohat}
T_{mix}^{1sub}(\hat{\rho}_0)\geq h_{max}\left(\frac{L^{d-1}}{E}\right)^{\frac{1}{2}}\frac{\sqrt{2}\pi}{4}.
\end{equation}

The analysis is split into two parts. We construct an explicit one-dimensional subsolution emanating from $\hat{\rho}_0$ which shows equality in \eqref{eq:lower_bound_for_rhohat}. This way we gain a clear picture of the vertical averages of density, velocity, momentum and energy of any potentially optimal mixing solution emanating from $\hat{\rho}_0$.

In order to obtain actual solutions realizing the potentially optimal mixing time \eqref{eq:lower_bound_smooth_for_rho0_hat} one then has to design the velocity field $v$ such that the above averages are ``undone''. This is a canonical and not too difficult task for convex integration which we utilize in a second step to finish the proof of Theorem \ref{thm:sharpness_weak_setting}. Of course it would be also very  interesting to see if other iteration techniques allow to weakly approach the subsolution by solutions with more regular velocity fields.

\subsection{Step 1: Sharpness on averaged level}\label{sec:sharpness_on_averaged_level}

By our investigation any one-dimensional subsolution $(\rho,m_1)$ that would show sharpness in \eqref{eq:bound_for_Tmix1} necessarily has to satisfy a list of properties:
\begin{itemize}
\item Recalling Remark \ref{rem:sharpness_first_bound} about steepest descent subsolutions the averaged momentum $m_1$ has to be given by \eqref{eq:sharp_choice_of_m_1D}, \eqref{eq:lambda_of_t}. In other words $\rho$ has to solve the nonlocal conservation law \eqref{eq:1D_nonlocal_conservation_law}.
\item If $\varphi$ denotes as usual the potential of $\rho$, then $w(t,\cdot):=\varphi_x(t,\cdot)$ has to satisfy 
\[
F(w(t,\cdot))=\sup_{w\in X_{h(t)}}F(w),\quad h(t):=\norm{w(t,\cdot)}_{L^2(\T)},\quad t\in(0,T),
\]
cf. Section \ref{sec:variational_problem}.
\item By Section \ref{sec:analysis_of_ode} the function $q(t):=h(t)/h_{max}$ has to satisfy \eqref{eq:ode_inequaltiy} with equality.
\end{itemize}
In the case of $\hat{\rho}_0$ as our initial data it is possible to explicitly verify these conditions.

In order to state the one-dimensional subsolution let $l:=L/2$ and define for $\alpha\in[0,1]$ the function $W^\alpha\in \Lip(\T)$  by setting first
\begin{align*}
W^\alpha(x):=\begin{cases}
x,&0\leq x\leq \alpha^2\frac{l}{2},\\
\alpha\frac{l}{2}\left(1-\frac{\left(1-\frac{x}{l/2}\right)^2}{1-\alpha^2}\right)^{\frac{1}{2}},&\alpha^2\frac{l}{2}<x<l-\alpha^2l/2,\\
l-x,&l-\alpha^2l/2\leq x\leq l, 
\end{cases}
\end{align*}
for $x\in [0,l]$, then extending $W^\alpha$ in an odd way onto $[-l,l]$, and then $L$-periodically. The parameter $\alpha$ will depend on time according to 
\[
\alpha(t):=S^{-1}\left(\frac{\pi}{4}-ct\right),\quad c:=2^{-\frac{1}{2}}E^{\frac{1}{2}}L^{-\frac{d-1}{2}}h_{max}^{-1},
\]
where $S^{-1}:[0,\pi/4]\rightarrow [0,1]$ is the inverse of $S$ defined in \eqref{eq:definition_of_S}.

We then define
\begin{align*}
w(t,x):=W^{\alpha(t)}(x),\quad \rho(t,x):=w_x(t,x),\quad m_1(t,x):=\lambda(t)(1-\rho(t,x)^2)w(t,x),
\end{align*}
where the factor $\lambda(t)$ is defined according to \eqref{eq:lambda_of_t}, i.e.
\begin{align}\label{eq:definition_of_lambda_sharp_subsolution}
\lambda(t):=\left(\frac{E}{L^{d-1}}\right)^{\frac{1}{2}}\left(\int_{\T}\big(1-w_x(t,x)^2\big)w(t,x)^2\:dx\right)^{-\frac{1}{2}}.
\end{align}

\begin{proposition}\label{prop:example_with_sharpness}
The just stated pair $(\rho,m_1)$ is a well-defined one-dimensional subsolution for $\hat{\rho}_0(\cdot -L/4)$ with energy density
\begin{equation}\label{eq:definition_of_e_in_sharp_subsolution}
e(t,x):=\lambda(t)^2(1-\rho(t,x)^2)w(t,x)^2.
\end{equation}
Moreover, $\norm{\rho(t,\cdot)}_{H^{-1}(\T)}\rightarrow 0$ as $t\rightarrow h_{max}E^{-\frac{1}{2}}L^{\frac{d-1}{2}}\sqrt{2}\pi/4$. In particular there holds equality in \eqref{eq:lower_bound_for_rhohat}.
\end{proposition}
\begin{proof} First of all we observe that as long as $\alpha<1$ not only $W^\alpha$ is Lipschitz, but also the derivative $W^\alpha_x$, which on $(\alpha^2l/2,l-\alpha^2l/2)$ is given by
\[
W^\alpha_x(x)=\alpha\left(1-\frac{\left(1-\frac{x}{l/2}\right)^2}{1-\alpha^2}\right)^{-\frac{1}{2}}\frac{1-\frac{x}{l/2}}{1-\alpha^2},
\]
cf. \eqref{eq:derivative_of_walpha}. It is easy to check that $\norm{W^\alpha_x}_{L^\infty(\T)}\leq 1$ for all $\alpha\in[0,1]$.

Next note that $S$ is indeed invertible on $[0,1]$ due to the fact that $S'(\alpha)=\sqrt{1-\alpha^2}$. As a consequence $\alpha$ is well-defined as a function in $\cC^0([0,T])\cap \cC^1(0,T)$ for $T:=\pi/(4c)$ and there holds  $\alpha(0)=1$ and $\alpha(T)=0$.
It follows that $\rho\in L^\infty((0,T)\times\T)$ with $\norm{\rho}_{L^\infty((0,T)\times\T)}\leq 1$, i.e. property \ref{itm:one_dim_subsol_i} of Definition \ref{def:one_dimensional_subsolution} holds true.

The definitions of $e$ in \eqref{eq:definition_of_e_in_sharp_subsolution} and $\lambda$ in \eqref{eq:definition_of_lambda_sharp_subsolution} directly imply property \ref{itm:one_dim_subsol_iii}. Also property \ref{itm:one_dim_subsol_iv} is immediate by the definition of $m_1$ and $e$. As a consequence we also have $m_1\in L^\infty(0,T;L^2(\T))$.

In order that $(\rho,m_1)$ is a one-dimensional subsolution emanating from $\hat{\rho}_0(\cdot-L/4)$ it thus only remains to verify property \ref{itm:one_dim_subsol_ii}, i.e. 
\begin{align}\label{eq:equation_for_rho_in_sharpness_proof}
\partial_t\rho +\lambda(t)\partial_x\big((1-\rho^2)w\big)=0,\quad \rho(0,\cdot)=\hat{\rho}_0(\cdot-L/4)
\end{align}
has to hold on $(0,T)\times\T$ in the sense of distributions. Since $\rho$ is continuous and piecewise smooth on $(0,T)\times\T$ and since $\rho(0,x)=W^{\alpha(0)}(x)=W^1(x)=\hat{\rho}_0(x-L/4)$ for almost every $x$, it is enough to verify \eqref{eq:equation_for_rho_in_sharpness_proof} pointwise on each region where $\rho$ is smooth. The equation trivially holds true on the space time regions where $W^{\alpha(t)}(x)$ is linear. By the odd symmetry it is thus enough to verify \eqref{eq:equation_for_rho_in_sharpness_proof} pointwise on the set $\Omega_\alpha:=\set{(t,x):\alpha(t)^2l/2<x<l-\alpha(t)^2l/2}$. Moreover, since by definition $\rho=w_x$, it is sufficient to show that $w$ solves the Hamilton-Jacobi equation
\begin{align}\label{eq:ham_jac_for_w}
\partial_tw+\lambda(t)(1-w_x^2)w=0
\end{align} 
on $\Omega_\alpha$.

In a first step towards \eqref{eq:ham_jac_for_w} straightforward computations show that $v(t,x):=W^{\beta(t)}(x)$ with $\beta(t):=e^{-t}$ is a solution of the rescaled Hamilton-Jacobi equation
\begin{align}\label{eq:ham_jac_for_v}
\partial_t v+\big(1-(v_x)^2\big)v=0
\end{align}
on the corresponding set $\Omega_\beta$. 

We then rewrite $w(t,x)=v(\tau(t),x)$ with $\tau(t):=-\log\alpha(t)$ such that \eqref{eq:ham_jac_for_v} on $\Omega_\beta$ implies
\begin{align}\label{eq:ham_jacobi_with_tau}
\partial_t w+\dot{\tau}(t)\big(1-(w_x)^2\big)w=0
\end{align}
on $\Omega_\alpha$.

Comparing with \eqref{eq:ham_jac_for_w} it therefore only remains to show that 
\begin{equation}\label{eq:tau_dot_equals_lambda}
\dot{\tau}(t)=\lambda(t).
\end{equation}
On one hand the definition of $\alpha(t)$ implies
\begin{align*}
\dot{\tau}(t)=-\frac{1}{\alpha(t)}\dot{\alpha}(t)=\frac{c}{\alpha(t)\sqrt{1-\alpha(t)^2}}.
\end{align*}
On the other hand we rewrite 
\[
\lambda(t)=\left(\frac{E}{L^{d-1}}\right)^{\frac{1}{2}}\left(\int_{\T}\big(1-w_x(t,x)^2\big)w(t,x)^2\:dx\right)^{-\frac{1}{2}}=\left(\frac{E}{L^{d-1}}\right)^{\frac{1}{2}}F(w(t,\cdot))^{-\frac{1}{2}}
\]
with $F$ defined in \eqref{eq:definition_of_F}. By Lemmas \ref{lem:supremum}, \ref{lem:reduced_var_problem}, \ref{lem:reduction_of_actual_variational_problem} we know that $w(t,\cdot)=W^{\alpha(t)}\in X_{h(t)}$ if and only if $h(t)=h_{max}\alpha(t)\sqrt{2-\alpha(t)^2}$ and that it maximizes $F$ in its class $X_{h(t)}$ with extremal value given by
\begin{align*}
\sup_{X_{h(t)}}F=h(t)^2-h_{max}^2\left(1-\sqrt{1-\frac{h(t)^2}{h_{max}^2}}\right)^2=2h_{max}^2\alpha(t)^2(1-\alpha(t)^2).
\end{align*}
We therefore conclude
\begin{align}\begin{split}\label{eq:reformulation_of_lambda}
\lambda(t)&=\left(\frac{E}{L^{d-1}}\right)^{\frac{1}{2}}F(w(t,\cdot))^{-\frac{1}{2}}=\frac{E^{\frac{1}{2}}}{\sqrt{2}L^{\frac{d-1}{2}}h_{max}}\frac{1}{\alpha(t)\sqrt{1-\alpha(t)^2}}\\
&=\frac{c}{\alpha(t)\sqrt{1-\alpha(t)^2}}.
\end{split}
\end{align}
This shows \eqref{eq:tau_dot_equals_lambda}. Hence $(\rho,m_1)$ is a one-dimensional subsolution for $\hat{\rho}_0(\cdot-L/4)$.

Finally, one observes that $\norm{\rho(t,\cdot)}_{H^{-1}(\T)}\rightarrow 0$ as $t\rightarrow T$ because $W^{\alpha(T)}=W^{0}=0$, and that
\[
T=\frac{\pi}{4c}=h_{max}\left(\frac{L^{d-1}}{E}\right)^{\frac{1}{2}}\frac{\sqrt{2}\pi}{4}.
\]
Thus $(\rho,m_1)$ is a perfectly mixing subsolution for $\hat{\rho}_0(\cdot-L/4)$ that realizes the lower bound \eqref{eq:lower_bound_for_rhohat} as the time of perfect mixing.
\end{proof}

\subsection{Step 2: Weak solutions close to the subsolution}\label{sec:convex_integration}
Before starting we remark that we can not directly apply the convex integration theorem stated in Appendix \ref{sec:convex_integration_appendix} to the sharp subsolution introduced in the previous subsection. One reason is that the energy $e$ becomes unbounded as $t\rightarrow 0$. Indeed, by \eqref{eq:reformulation_of_lambda} and $\alpha(0)=1$ there holds
\begin{align*}
\sup_x e(t,x)&=\lambda(t)^2\sup_x(1-W^{\alpha(t)}_x(x)^2)W^{\alpha(t)}(x)^2=\lambda(t)^2W^{\alpha(t)}(l/2)^2\\
&=\frac{c^2}{\alpha(t)^2(1-\alpha(t)^2)}\alpha(t)^2l^2/4\sim \frac{1}{1-\alpha(t)^2}\rightarrow\infty.
\end{align*}
Thus the subsolution leaves any of the compact sets $K_{e\leq \gamma}:=K\cap\set{e\leq \gamma}$, $\gamma>0$ for which convex integration is carried out in Theorem \ref{thm:convex_integration}. This can be circumpassed by successively convex integrating on sets approaching $t=0$.

Another reason is that the subsolution from Section \ref{sec:sharpness_on_averaged_level} is not a strict subsolution in the sense that it takes values in the boundary of the convex hull instead of its interior. For that we will introduce a small gap leading to a mixing time exceeding the optimal one by an arbitrarily small error.

\begin{proof}[Proof of Theorem \ref{thm:sharpness_weak_setting}.]
In order to introduce the small gap we recall that $W^\alpha$, $\alpha\in[0,1]$ is independent of $E$, whereas the time dependent function $\alpha(t)=S^{-1}(\pi/4-ct)$ depends on $E$ through $c=2^{-\frac{1}{2}}E^{\frac{1}{2}}L^{-\frac{d-1}{2}}h_{max}^{-1}$. We indicate this dependence by writing $\alpha_E(t)$. Similarly we write $\lambda_{E,w}(t)$ to indicate the dependencies of
\[
\lambda(t)=\left(\frac{E}{L^{d-1}}\right)^{\frac{1}{2}}\left(\int_{\T}\big(1-w_{x_1}(t,x_1)^2\big)w(t,x_1)^2\:dx_1\right)^{-\frac{1}{2}}.
\] 

Let now $\varepsilon\in(0,E/2)$ and set
\[
T_\varepsilon:=h_{max}\left(\frac{L^{d-1}}{E-2\varepsilon}\right)^{\frac{1}{2}}\frac{\sqrt{2}\pi}{4}.
\]
We define the tuple $(0,T_\varepsilon)\times\T^d\ni (t,x)\mapsto z(t,x)=(\rho,v,m,e)(t,x)\in\R^{1+d+d+1}$ by
\begin{align*}
w(t,x)&:=W^{\alpha_{E-2\varepsilon}(t)}(x_1),\\
\rho(t,x)&:=w_{x_1}(t,x),\\
v(t,x)&:=0,\\
m(t,x)&:=\lambda_{E-2\varepsilon,w}(t)(1-w_{x_1}(t,x)^2)w(t,x)e_1,\\
e(t,x)&:=\lambda_{E-\varepsilon,w}(t)^2(1-w_{x_1}(t,x)^2)w(t,x)^2.
\end{align*}

Then $\rho\in L^\infty((0,T_\varepsilon)\times \T^d)\cap \cC^0([0,T_\varepsilon];L^2_{weak}(\T^d))$, $m\in L^\infty(0,T_\varepsilon;L^2(\T^d;\R^d))$, $e\in L^\infty(0,T_\varepsilon;L^1(\T^d))$ with
\begin{align}\label{eq:integral_e_in_convex_integration_application}
\int_{\T^d}e(t,x)\:dx=L^{d-1}\int_{\T}e(t,x_1)\:dx_1=E-\varepsilon.
\end{align}
for all $t\in(0,T_\varepsilon)$. By Proposition \ref{prop:example_with_sharpness} we also know that $z$ satisfies \eqref{eq:linear_system_ci} on $(0,T_\varepsilon)\times\T^d$ with $\rho_0(x)=\hat{\rho}_0(x_1-L/4)$ in the sense of distributions.

The associated mixing zone $\mathscr{U}$ is given by
\[
\mathscr{U}:=\set{(t,x)\in(0,T)\times\T^d:\alpha_{E-2\varepsilon}(t)^2l/2<\abs{x_1}<l-\alpha_{E-2\varepsilon}(t)^2l/2}.
\]
Indeed outside $\mathscr{U}$ we have $\abs{\rho}=1$, $m=\rho v=0$ and $e=\abs{v}^2=0$, i.e. $z(t,x)\in K$, whereas inside $\mathscr{U}$ there holds $\abs{\rho}<1$ and 
\begin{align*}
\abs{m(t,x)}^2&=\lambda_{E-2\varepsilon,w}(t)^2(1-\rho(t,x)^2)^2w(t,x)^2\\
&<\lambda_{E-\varepsilon,w}(t)^2(1-\rho(t,x)^2)^2w(t,x)^2=e(t,x)(1-\rho(t,x)^2).
\end{align*} 
Thus $z(t,x)\in \text{int}(K^{co})$ for all $(t,x)\in\mathscr{U}$. Moreover, inside $\mathscr{U}$ by \eqref{eq:reformulation_of_lambda} there also holds
\begin{align*}
\frac{4\abs{m(t,x)}^2}{(1-\rho(t,x)^2)^2}&+e(t,x)<5\lambda_{E,w}(t)^2w(t,x)^2\leq 5\lambda_{E,w}(t)^2\alpha_{E-2\varepsilon}(t)^2l^2/4\\
&= \frac{5E}{4(E-2\varepsilon)}\lambda_{E-2\varepsilon,w}(t)^2\alpha_{E-2\varepsilon}(t)^2l^2=\frac{C}{1-\alpha_{E-2\varepsilon}(t)^2}=:\gamma(t)
\end{align*}
for a suitable constant $C=C(E,\varepsilon,L)>0$. Hence Lemma \ref{lem:convex_hull_refinement} implies $z(t,x)\in \text{int}(K^{co}_{e\leq \gamma(t)})$ for all $(t,x)\in \mathscr{U}$.

Since outside $\mathscr{U}$ there also holds $e(t,x)\leq \gamma(t)$, we have $z(t,x)\in K_{e\leq \gamma(t)}$, $(t,x)\notin\mathscr{U}$ as well.

We now apply Theorem \ref{thm:convex_integration} on increasing subsets of $\mathscr{U}$, cf. Remark \ref{rem:remarks_to_convex_integration_theorem} d). This has been done in a similar way for instance in \cite[Appendix B]{GHK_LAP}.
Let $t_j\searrow 0$ and define $\mathscr{U}_j:=\mathscr{U}\cap (t_j,T_\varepsilon)\times\T^d$. In the first step for instance we convex integrate on $\mathscr{U}_0$ and with the compact set $K_{e\leq \gamma(t_0)}$. Note here that $\gamma(t)$ is monotone and also that $e\in\cC^0([t_0/2,T_\varepsilon];L^2_{weak}(\T^d))$. Theorem \ref{thm:convex_integration}, see also Remarks \ref{rem:remarks_to_convex_integration_theorem} b), c), d), gives us infinitely many new strict subsolutions $z^{(1)}$ which coincide with $z$ outside $\mathscr{U}_0$, which now satisfy $\abs{\rho^{(1)}}=1$, $m^{(1)}=\rho^{(1)} v^{(1)}$, $\abs{v^{(1)}}^2=e^{(1)}$ almost everywhere on $\mathscr{U}_0$ and which satisfy $\norm{\rho^{(1)}(t,\cdot)}_{H^{-1}(\T^d)}\rightarrow 0$ as $t\rightarrow T_\varepsilon$. In view of \eqref{eq:integral_e_in_convex_integration_application} we can also pick the new subsolutions such that
\begin{align*}
\int_{\T^d}e^{(1)}(t,x)\:dx\leq \begin{cases}E-\varepsilon,& t\in(0,t_0],\\
E,&t\in(t_0,T_\varepsilon].
\end{cases}
\end{align*}

Inductively we continue to convex integrate on $\mathscr{U}_j\setminus \mathscr{U}_{j-1}$ to obtain (infinitely many) sequences $(z^{(j)})_j$ which take values in $K_{e\leq \gamma(t_{j-1})}$ almost everywhere on $(t_{j-1},T_\varepsilon)\times\T^d$ and that coincide with the original subsolution on $(0,t_{j-1})\times\T^d$. Their total energy satisfies 
\begin{align*}
\int_{\T^d}e^{(j)}(t,x)\:dx\leq E\text{ for all }t\in(0,T). 
\end{align*}
This uniform bound together with $\norm{\rho^{(j)}}_{L^\infty}\leq 1$ induces the uniform bounds $\norm{v^{(j)}}_{L^\infty_tL^2_x}\leq \sqrt{E}$, $\norm{m^{(j)}}_{L^\infty_tL^2_x}\leq \sqrt{E}$ which allow to find subsequences such that $\rho^{(j)}\rightharpoonup \tilde{\rho}$, $v^{(j)}\rightharpoonup \tilde{v}$, $m^{(j)}\rightharpoonup \tilde{m}$ weakly in $L^2((0,T_\varepsilon)\times\T^d)$. Since the convergence also holds true pointwise by the iterative nature of the construction, it is easy to see that the limiting pair $(\tilde{\rho},\tilde{v})$ is a distributional solution of \eqref{eq:equations_distributional} on $(0,T_\varepsilon)\times\T^d$ with initial data $\hat{\rho}_0(\cdot-L/4)$ and with $\norm{\tilde{v}}_{L^\infty(0,T_\varepsilon;L^2(\T^d))}\leq \sqrt{E}$, $\abs{\tilde{\rho}}=1$ almost everywhere on $(0,T_\varepsilon)\times\T^d$. By construction $\tilde{\rho}$ which coincides with $\rho^{(1)}$ near $\{T_\varepsilon\}\times\T^d$ is perfectly mixed at $T_\varepsilon$. Observing that $T_\varepsilon\searrow h_{max}L^{\frac{d-1}{2}}E^{-\frac{1}{2}}\sqrt{2}\pi/4=\hat{T}_0$ finishes the proof of Theorem \ref{thm:sharpness_weak_setting}.
\end{proof}

\appendix

\section{\texorpdfstring{Lipschitz continuity of the $H^{-1}$-norm}{Lipschitz continuity of the mixing norm}}\label{sec:appendix_lipschitz}

Let $\rho\in L^\infty((0,T)\times\T^d)$ and $m\in L^\infty(0,T; L^2(\T^d;\R^d))$ be a solution in the sense of distributions of
\begin{align}\label{eq:linear_equation_appendix}
    \partial_t\rho+\divv m=0,\quad \rho(0,\cdot)=\rho_0
\end{align}
for some initial data $\rho_0$ satisfying \eqref{eq:rho_0_conditions}. In view of \cite[Lemma 8]{DeLellis_Sz_2010} we will assume that $\rho\in\cC^0([0,T);L^2_w(\T^d))$, meaning that \[
[0,T)\ni t\mapsto \int_{\T^d}\rho(t,x)\psi(x)\:dx\in \R
\] 
is continuous for any fixed $\psi\in L^2(\T^d)$. In particular we have that
\begin{align*}
    \norm{\rho(t,\cdot)}_{H^{-1}(\T^d)}=\norm{\nabla \varphi(t,\cdot)}_{L^2(\T^d)},
\end{align*}
where $\varphi(t,\cdot)$ is defined through $\Delta \varphi(t,\cdot)=\rho(t,\cdot)$, is well-defined for all $t\in[0,T)$. For the definition of $\varphi$ note that $\fint \rho(t,\cdot)=0$ for all $t\in (0,T)$ due to $\fint\rho_0=0$ and  \eqref{eq:linear_equation_appendix}.
Also note that $\varphi\in L^\infty(0,T;W^{2,p}(\T^d))$ for any $p\in [1,\infty)$.

\begin{lemma}\label{lem:mixing_norm_Lipschitz}
    In the just described situation the $H^{-1}$-norm squared $\norm{\rho(t,\cdot)}_{H^{-1}(\T^d)}^2$ is Lipschitz continuous on $[0,T)$ and for almost every $t\in(0,T)$ there holds
    \begin{align}\label{eq:derivative_mixing_norm}
        \frac{d}{dt}\norm{\rho(t,\cdot)}_{H^{-1}(\T^d)}^2=-2\int_{\T^d}m(t,x)\cdot\nabla\varphi(t,x)\:dx.
    \end{align}
\end{lemma}
\begin{proof} We follow \cite[Lemma 3.1]{Fjordholm_Lanthaler_Mishra_2017}.
    Let $(\eta_\varepsilon)_{\varepsilon>0}$ be a family of symmetric one-dimensional mollifiers and $\chi\in\cC^\infty_c([0,T))$. Consider $\psi_\varepsilon,\tilde\psi_\varepsilon:[0,T)\rightarrow \R$,
    \begin{align*}
        \psi_\varepsilon(t,x)&=\int_0^T\eta(t-s)\chi(s)\varphi(s,x)\:ds,\\
        \tilde{\psi}_\varepsilon(t,x)&=\chi(t)\int_0^T\eta_\varepsilon(t-s)\varphi(s,x)\:ds.
    \end{align*}
    For $\varepsilon>0$ small enough these functions are, after approximation with smooth functions, valid testfunctions for \eqref{eq:linear_equation_appendix}. Thus
    \begin{align*}
        \int_0^T&\int_{\T^d}\rho(\partial_t\psi_\varepsilon+\partial_t\tilde\psi_\varepsilon)\:dx\:dt=-\int_0^T\int_{\T^d}m\cdot(\nabla\psi_\varepsilon+\nabla\tilde{\psi}_\varepsilon)\:dx\:dt\\
        &\hspace{195pt}-\int_{\T^d}\rho_0(\psi_\varepsilon(0,\cdot)+\tilde\psi_\varepsilon(0,\cdot))\:dx\\
        &=-\int_0^T\int_0^T\int_{\T^d}m(t,x)\cdot\nabla\varphi(s,x)\eta_\varepsilon(t-s)(\chi(s)+\chi(t))\:dx\:ds\:dt\\
        &\hspace{45pt}-\int_0^T\int_{\T^d}\rho_0(x)\eta_\varepsilon(-s)\varphi(s,x)(\chi(s)+\chi(0))\:dx\:ds.
    \end{align*}
    On the other hand, computing the time derivatives one sees that 
    \begin{align*}
        \int_0^T\int_{\T^d}&\rho(\partial_t\psi_\varepsilon+\partial_t\tilde\psi_\varepsilon)\:dx\:dt=\int_0^T\int_0^T\int_{\T^d}\rho(t,x)\chi'(t)\eta_\varepsilon(t-s)\varphi(s,x)\:dx\:ds\:dt\\
        &+\int_0^T\int_0^T\int_{\T^d}\rho(t,x)\eta_\varepsilon'(t-s)(\chi(s)+\chi(t))\varphi(s,x)\:dx\:ds\:dt.
    \end{align*}
    Now $\rho(t,x)=\Delta\varphi(t,x)$ and partial integration in $x$ show that the latter term in fact vanishes due to the antisymmetry of $\eta_\varepsilon'$. Combininge the last two identities we therefore obtain
    \begin{align*}
        \int_0^T\int_0^T&\int_{\T^d}\rho(t,x)\varphi(s,x)\eta_\varepsilon(t-s)\chi'(t)\:dx\:ds\:dt\\
        &=-\int_0^T\int_0^T\int_{\T^d}m(t,x)\cdot\nabla\varphi(s,x)\eta_\varepsilon(t-s)(\chi(s)+\chi(t))\:dx\:ds\:dt\\
        &\hspace{45pt}-\int_0^T\int_{\T^d}\rho(s,x)\eta_\varepsilon(-s)\varphi(0,x)(\chi(s)+\chi(0))\:dx\:ds.
    \end{align*}
    
    Since $\rho\in\cC^0([0,T);L^2_w(\T^d))$ we can pass to the limit $\varepsilon\rightarrow 0$. This is clear for the left-hand side and the last term. For the first term on the right-hand side we use the Leray projection onto divergence-free vector fields $\P=\id-\nabla\Delta^{-1}\divv$ in order to write 
    \begin{align*}
    \int_{\T^d}m(t,\cdot)\cdot\nabla\varphi(s,\cdot)\:dx&=\int_{\T^d}(\id-\P)[m(t,\cdot)]\cdot\nabla\varphi(s,\cdot)\:dx\\
    &=-\int_{\T^d}\Delta^{-1}\divv m(t,\cdot)\rho(s,\cdot)\:dx
    \end{align*}
    which is continuous in $s$. Undoing this reformulation after taking the limit $\varepsilon\rightarrow 0$ and using again $\Delta\varphi=\rho$ together with partial integration we deduce
    \begin{align*}
        -\int_0^T\int_{\T^d}\abs{\nabla\varphi(t,x)}^2\chi'(t)\:dx\:dt=-2\int_0^T\int_{\T^d}m(t,x)&\cdot\nabla\varphi(t,x)\chi(t)\:dx\:dt\\
        &+\int_{\T^d}\abs{\nabla\varphi(0,x)}^2\chi(0)\:dx.
    \end{align*}
    Thus $\norm{\rho(t,\cdot)}_{H^{-1}(\T^d)}^2$ has the weak derivative stated in \eqref{eq:derivative_mixing_norm}. Finally recall that $m\in L^\infty(0,T;L^2(\T^d;\R^d))$ and $\varphi\in L^{\infty}(0,T;W^{2,2}(\T^d))$, which is (more than) enough to conclude the statement of the lemma.
\end{proof}

\section{A convex hull}\label{sec:appendix_convex_hull}

Recall that $K\subset\R\times\R^d\times\R^d\times\R$ appearing in \eqref{eq:definition_of_K} is defined by $(\rho,v,m,e)\in K$ if and only if
\begin{align*}
\abs{\rho}\leq 1,\quad m=\rho v,\quad \abs{v}^2=e.
\end{align*}
If $e$ would be considered as a fixed quantity instead of a variable, then $K$ and its convex hull computed in Lemma \ref{lem:convex_hull_appendix} essentially could be seen as a simplified case of the convex hulls computed in \cite{GK_Boussinesq,GKSz_RT} with Kolumb\'an, Kolumb\'an and Sz\'ekelyhidi resp., in the context of the inhomogeneous Euler equations.
\begin{lemma}\label{lem:convex_hull_appendix}
A point $(\rho,v,m,e)\in\R\times\R^d\times\R^d\times\R$ belongs to the closed convex hull of $K$ if and only if one of the following two conditions is satisfied
\begin{enumerate}[(i)]
\item\label{itm:convex_hull1} $\abs{\rho}=1$, $m=\rho v$ and $\abs{v}^2\leq e$,
\item\label{itm:convex_hull2} $\abs{\rho}<1$ and $\abs{m-\rho v}^2\leq (e-\abs{v}^2)(1-\rho^2)$.
\end{enumerate}
\end{lemma}
\begin{proof}
Let us denote the set defined through \ref{itm:convex_hull1}, \ref{itm:convex_hull2} by $C_1$, $C_2$ and set $C:=C_1\cup C_2$. Clearly $K\subset C$. In order to see that $C$ is a convex set we define for $v,m\in\R^d$, $\rho\in(-1,1)$ the symmetric matrix
\begin{align*}
M_{\rho,v,m}:=\frac{m\otimes m-\rho (v\otimes m+ m\otimes v)+v\otimes v}{1-\rho^2}\in\R^{d\times d}
\end{align*}
and observe that for an arbitrary orthonormal base $\xi_1,\ldots,\xi_d$ of $\R^d$ there holds
\begin{align*}
\frac{\abs{m-\rho v}^2}{1-\rho^2}+\abs{v}^2=\tr(M_{\rho,v,m})=\sum_{i=1}^d\xi_i^T M_{\rho,v,m}\xi_i.
\end{align*}
In \cite[Proof of Lemma 3.7]{GK_Boussinesq} it has been shown that the function
\begin{align*}
(-1,1)\times\R^d\times\R^d\ni(\rho,v,m)\mapsto \xi^TM_{\rho,v,m}\xi\in\R
\end{align*}
considered with an arbitrary fixed $\xi\in\R^d$, $\abs{\xi}=1$ is convex. Thus the set $C_2$ is convex, as is its closure $\overline{C}_2$. Since $\overline{C}_2=C_2\cup C_1=C$. We find that $K^{co}\subset C$.

It remains to show that $C\subset K^{co}$. Note that we can not argue by means of the Krein-Milman Theorem as $K$ is not a compact set. Instead we show the inclusion via iteration of segments. In a first step we claim that 
\[
K_1:=\set{(\rho,v,m,e)\in\R\times\R^d\times\R^d\times\R:\abs{\rho}<1,~\abs{m-\rho v}^2=(e-\abs{v}^2)(1-\rho^2)}
\]
is a subset of $K^{co}$. Indeed let $z:=(\rho,v,m,e)\in K_1$ and define $\hat{z}=(\hat{\rho},\hat{v},\hat{m},\hat{e})$ by setting
\begin{gather*}
\hat{\rho}:=1,\quad \hat{v}:=\frac{m-\rho v}{1-\rho^2},\quad
\hat{m}:=\frac{v-\rho m}{1-\rho^2},\quad \hat{e}:=\frac{\abs{m+v}^2}{(1-\rho^2)(1+\rho)}-\frac{e}{1-\rho}.
\end{gather*}
Straightforward computations show that $z+s\hat{z}\in K$ for $s=1-\rho>0$ and $s=-1-\rho<0$. Consequently $z\in K^{co}$.

Next, using a second order segment, only changing the $m$-component for instance, it follows that $C_2\subset K_1^{co}\subset K^{co}$. On the other hand $C_1\subset K^{co}$ is clear. Alltogether we have shown $C=K^{co}$.
\end{proof}

\begin{proof}[Proof of Lemma \ref{lem:convex_hull_with_zero_velocity}] Let $z\in K^{co}$ with $v_1=0$. If $\abs{\rho}=1$, then $m_1=\rho v_1=0$ and $m_1^2\leq e(1-\rho^2)$ trivially holds true. If $\abs{\rho}<1$, then
\begin{align*}
m_1^2&=(m_1-\rho v_1)^2\leq \abs{m-\rho v}^2+\abs{v}^2(1-\rho^2)\leq e(1-\rho^2).
\end{align*}
This proves \eqref{eq:special_hull_inequality}.
\end{proof}

\section{\texorpdfstring{$h$-principle for energy constrained mixing}{h-principle for energy constrained mixing}}\label{sec:convex_integration_appendix}
The following two statements are used in Section \ref{sec:convex_integration} for the construction of actual mixing solutions via convex integration. The first is a refinement of Lemma \ref{lem:convex_hull_appendix} giving a sufficient criterion under which certain states belong to the hull of a compact subset $K_{e\leq \gamma}$ of the unbounded set $K$. Thereafter we state the convex integration theorem. Note that its proof is just relying on the form of $K$ and the compactness of $K_{e\leq \gamma}$, in particular it does not require a full, in fact not even a partial, description of $K_{e\leq \gamma}^{co}$. 
\begin{lemma}\label{lem:convex_hull_refinement}
Let $\gamma>0$ and suppose that $(\rho,0,m,e)\in \text{int}(K^{co})$ with
\begin{align}\label{eq:extra_condition_bounded_hull}
\frac{4\abs{m}^2}{(1-\rho^2)^2}+e< \gamma,
\end{align}
then $z\in\text{int}(K^{co}_{e\leq\gamma})$ where $K_{e\leq\gamma}$ is defined as $K\cap \set{(\rho,v,m,e):e\leq \gamma}$.
\end{lemma}
\begin{proof}
We repeat the first segment of the proof of Lemma \ref{lem:convex_hull_appendix} and claim that all tuple $(\rho,v,m,e)\in K_1$ with \begin{align}\label{eq:gamma_convex_hull_condition1}
\abs{\frac{m+v}{1+\rho}}^2\leq \gamma,\quad \abs{\frac{m-v}{1-\rho}}^2\leq \gamma
\end{align}
belong to the closed convex hull of $K_{e\leq\gamma}$. Indeed, we had seen that the endpoints of the segment $z+s\hat{z}$, $s\in[-1-\rho,1-\rho]$ are in $K$. Thus they are in $K_{e\leq\gamma}$ provided $e+(1-\rho)\hat{e}\leq \gamma$ and $e-(1+\rho)\hat{e}\leq \gamma$. But quick computations show
\begin{align*}
e+(1-\rho)\hat{e}=\abs{\frac{m+v}{1+\rho}}^2,\quad e-(1+\rho)\hat{e}=\abs{\frac{m-v}{1-\rho}}^2.
\end{align*}

Let now $z_0=(\rho,0,m,e)$ be in the interior of $K^{co}$ which by Lemma \ref{lem:convex_hull_appendix} means $\abs{\rho}<1$ and $\abs{m}^2<e(1-\rho^2)$. We will construct a segment containing $z_0$ and with endpoints belonging to $K_1$ and in addition satisfying \eqref{eq:gamma_convex_hull_condition1} provided condition \eqref{eq:extra_condition_bounded_hull} holds true. Let $\hat{z}:=(0,\hat{v},\rho\hat{v},0)$ for some $\hat{v}\in \R^d$ with $\abs{\hat{v}}=1$, $\hat{v}\cdot m=0$. It is easy to check that $z_0+s\hat{z}\in K_1$ if and only if
\[
s^2=e-\frac{\abs{m}^2}{1-\rho^2}.
\]
Moreover, for those two $s$ there holds
\begin{align*}
\abs{\frac{m+s\hat{m}+s\hat{v}}{1+\rho}}^2&=\abs{\frac{m}{1+\rho}+s\hat{v}}^2=\abs{\frac{m}{1+\rho}}^2+s^2=e-\frac{2\rho(1-\rho)}{(1-\rho^2)^2}\abs{m}^2\\
&\leq e+\frac{4\abs{m}^2}{(1-\rho^2)^2}< \gamma
\end{align*}
by \eqref{eq:extra_condition_bounded_hull}. In a similar way
\begin{align*}
\abs{\frac{m+s\hat{m}-s\hat{v}}{1-\rho}}^2=e+\frac{2\rho(1+\rho)}{(1-\rho^2)^2}\abs{m}^2\leq e+\frac{4\abs{m}^2}{(1-\rho^2)^2}< \gamma.
\end{align*}
It follows that $z_0\in \text{int}(K^{co}_{e\leq\gamma})$.
\end{proof}

\begin{theorem}\label{thm:convex_integration}
Let $\gamma>0$, $T>0$ and $\rho_0\in L^\infty(\T^d)$ satisfying \eqref{eq:rho_0_conditions}. Suppose that $z:=(\rho,v,m,e)\in L^\infty((0,T)\times\T^d;\R^{1+d+d+1})$ with $\rho,e\in \cC^0([0,T];L^2_{weak}(\T^d))$ solves
\begin{align}\label{eq:linear_system_ci}
\begin{cases}
\partial_t\rho+\divv m=0,&\text{on }(0,T)\times\T^d,\\
\divv v=0,&\text{on }(0,T)\times\T^d,\\
\rho(0,\cdot)=\rho_0,&\text{on }\T^d,\\
\rho(T,\cdot)=0,&\text{on }\T^d
\end{cases}
\end{align}
in the sense of distributions. Suppose further that there exists $\mathscr{U}\subset (0,T)\times\T^d$ open with $z_{|\mathscr{U}}$ continuous, $z(t,x)\in \text{int}(K^{co}_{e\leq\gamma})$ for all $(t,x)\in\mathscr{U}$ and with $z(t,x)\in K_{e\leq \gamma}$ for almost every $(t,x)\in(0,T)\times\T^d\setminus \mathscr{U}$. Then there exist infinitely many tuple $z_{sol}=(\rho_{sol},v_{sol},m_{sol},e_{sol})\in L^\infty((0,T)\times\T^d;\R^{1+d+d+1})$ with $\rho_{sol},e_{sol}\in \cC^0([0,T];L^2_{weak}(\T^d))$ solving \eqref{eq:linear_system_ci} distributionally and satisfying $z_{sol}=z$ almost everywhere on $(0,T)\times\T^d\setminus\mathscr{U}$, $z_{sol}\in K_{e\leq \gamma}$ almost everywhere on $\mathscr{U}$.
Moreover, infinitely many of the above solutions $z_{sol}$ can be found arbitrarily close to $z$ with respect to the (metrizable) topology of
$\cC^0([0,T];L^2_{weak}(\T^d))\times L^2_{weak}((0,T)\times\T^d;\R^{d+d})\times \cC^0([0,T];L^2_{weak}(\T^d))$.
\end{theorem}
\begin{remark}\label{rem:remarks_to_convex_integration_theorem}
a) We call any tuple $z$ as above a strict subsolution for $\rho_0$ and the set $\mathscr{U}$ the mixing zone of $z$.

b) The induced solutions in particular satisfy
\begin{align*}
\norm{\rho_{sol}(t,\cdot)}_{H^{-1}(\T^d)}\rightarrow 0\text{ as }t\rightarrow T.
\end{align*}

c) By the closeness of the $e$-component in $\cC^0([0,T];L^2_{weak}(\T^d))$ one in particular finds for any $\varepsilon>0$ infinitely many solutions $z_{sol}$ satisfying
\begin{align*}
\int_{\T^d}\abs{v_{sol}(t,x)}^2\:dx=\int_{\T^d}e_{sol}(t,x)\:dx\leq \int_{\T^d}e(t,x)\:dx+\varepsilon
\end{align*}
for all $t\in[0,T]$.

d) By modifying the definition of $X_0$ in the proof below one can also convex integrate not on all of $\mathscr{U}$, but instead on an open subset $\mathscr{V}$ of $\mathscr{U}$. The result are infinitely many strict subsolutions with mixing zone $\mathscr{U}\setminus\mathscr{V}$ where they still coincide with the original subsolution. 
\end{remark}

\begin{proof}[Proof of Theorem \ref{thm:convex_integration}.] The proof is relying on convex integration via a Baire category argument in the Tartar framework which is by now rather standard and therefore only sketched on a rough level, see \cite{DeLellis_Sz_2009,DeLellis_Sz_2010} for the origins in fluid dynamics and 
\cite{Castro_Faraco_Mengual_degraded,Choffrut_Sz_stationary,Cordoba_Faraco_Gancedo_lack_of_uniqueness,
Crippa_Gusev_Spirito_Wiedemann,Faraco_Lindberg_Sz_bounded_mhd,Faraco_Lindberg_Sz_second_paper,GK_EE,GK_Boussinesq,GKSz_RT,
Markfelder_EE,Markfelder_thesis,Mengual_different_mobilities,Sz_ipm} for extensions and applications in various fluid equations, see also \cite{Sattig_Sz_baire_category} for a Baire category argument outside the Tartar framework.  Especially also since the here considered differential inclusion is simpler than the ones for the inhomogeneous Euler equations \cite{GK_Boussinesq,GKSz_RT}.

The Baire category argument is usually carried out in the Tartar framework by verifying three properties (H1)-(H3) for the differential inclusion, see \cite[Appendix]{Sz_ipm}.

Property (H1) addresses the existence of localized plane waves as perturbation elements. In our case this is simpler as in the Euler case \cite{GK_Boussinesq,GKSz_RT} and therefore omitted. We only point out that one can upgrade the convergence for the $\rho$ and $e$ components from weak-$L^2_{t,x}$ convergence to uniform in time weak-$L^2_x$ convergence, see for example \cite[Remark 3.3]{GK_Boussinesq}, but of course also \cite{DeLellis_Sz_2010}.

Property (H2) is a geometric perturbation lemma allowing to perturb along sufficiently long segments. It is known \cite{DeLellis_Sz_2010,GK_Boussinesq,GKSz_RT} that this is automatically satisfied provided the wave cone $\Lambda$ is large enough with respect to the set of constraints, meaning that (H2) holds true under the condition that for all $z_1,z_2\in K_{e\leq \gamma}$, $z_1\neq z_2$ there holds $z_2-z_1\in \Lambda$. In our case we have
\begin{align*}
\Lambda:=\set{\hat{z}=(\hat{\rho},\hat{v},\hat{m},\hat{e}):\exists (\xi,c)\in\R^d\times\R\setminus\{0\},~c\hat{\rho}+\xi\cdot\hat{m}=0,~\xi\cdot \hat{v}=0}.
\end{align*}
Thus in dimension $d\geq 3$ there holds $\Lambda=\R^{1+d+d+1}$ and for $d=2$ there holds 
\[
\Lambda=\R^{1+2+2+1}\setminus\set{\hat{z}:\hat{\rho}=0,~\hat{v}\text{ not parallel to }\hat{m}}.
\]
But also in the 2D-case it is easy to verify that $z_1,z_2\in K_{e\leq \gamma}$, $z_1\neq z_2$ implies $z_2-z_1\in \Lambda$. Thus (H2) holds true.

Finally (H3) addresses the metrizability of the weak $L^2$-topology. More precisely, let $z$ denote the strict subsolution of Theorem \ref{thm:convex_integration} and define the set $X_0$ as the set of all strict subsolutions that have the same mixing zone $\mathscr{U}$ as $z$ and that coincide with $z$ outside $\mathscr{U}$. Since by definition all these subsolutions take values in the closed convex hull of $K_{e\leq \gamma}$ which is a compact set in $\R^{1+d+d+1}$, the set $X_0$ is bounded in $\cC^0_tL^2_x\times L^2_{t,x}\times L^2_{t,x}\times \cC^0_tL^2_x$. Thus the corresponding weak topology $\cC^0_tL^2_{x,weak}\times L^2_{t,x,weak}\times L^2_{t,x,weak}\times \cC^0_tL^2_{x,weak}$ is metrizable by a metric $d$. This is (H3).

Let $X$ denote the closure of $X_0$ with respect to $d$. The three properties allow to conclude that the subset $\set{\tilde{z}\in X:\tilde{z}(t,x)\in K_{\gamma}\text{ for almost every }(t,x)\in \mathscr{U}}$ is residual in the metric space $(X,d)$.
\end{proof}

\bibliographystyle{abbrv}
\bibliography{optimal_mixing}

\vspace{5pt}
\noindent\textbf{Acknowledgements.} Funded by the Deutsche Forschungsgemeinschaft (DFG, German Research Foundation) under Germany's Excellence Strategy EXC 2044-390685587, Mathematics Münster: Dynamics-Geometry-Structure.
In parts the research has also been carried out at Universidad Aut\'onoma de Madrid and Max Planck Institute for Mathematics in the Sciences, where the author acknowledges the support through Mar\'ia Zambrano Grant CA6/RSUE/2022-00097 and the great hospitality of MPI Leipzig.
The author also would like to thank Christian Seis and \'Angel Castro for fruitful discussions on the paper.

\vspace{10pt}

\noindent Bj\"orn Gebhard\\
Mathematics M\"unster, University of M\"unster, 48149 Münster, Germany\\
bjoern.gebhard@uni-muenster.de

\end{document}